\documentclass{amsart}%
\usepackage{amsmath}
\usepackage{amssymb}
\usepackage{amsfonts}
\usepackage[all,cmtip]{xy}
\usepackage[backref]{hyperref}
\usepackage[margin=1.39in]{geometry}
\usepackage{graphicx}%
\newtheorem{theorem}{Theorem}
\newtheorem*{theoremannounce}{Theorem}
\numberwithin{theorem}{section}
\theoremstyle{plain}
\newtheorem*{acknowledgement}{Acknowledgement}

\newtheorem{corollary}[theorem]{Corollary}

\newtheorem{definition}[theorem]{Definition}

\newtheorem{lemma}[theorem]{Lemma}
\newtheorem*{notation}{Notation}

\newtheorem{proposition}[theorem]{Proposition}
\theoremstyle{remark}
\newtheorem{example}{Example}
\newtheorem{remark}[theorem]{Remark}

\numberwithin{equation}{section}
\begin{document}
\title[$K$-theory of locally compact modules]{$K$-theory of locally compact modules\linebreak over rings of integers}
\author{Oliver Braunling}
\address{Freiburg Institute for Advanced Studies (FRIAS), University of Freiburg,
D-79104 Freiburg im\ Breisgau, Germany}
\thanks{The author was supported by DFG GK1821 \textquotedblleft Cohomological Methods
in Geometry\textquotedblright\ and a Junior Fellowship at the Freiburg
Institute for Advanced Studies (FRIAS)}

\begin{abstract}
We generalize a recent result of Clausen:\ For a number field with integers
$\mathcal{O}$, we compute the $K$-theory of locally compact $\mathcal{O}%
$-modules. For the rational integers this recovers Clausen's result as a
special case. Our method of proof is quite different:\ Instead of a
homotopy coherent cone construction in $\infty$-categories, we rely on
calculus of fraction type results in the style of Schlichting. This produces
concrete exact category models for certain quotients, a fact which might be of
independent interest. As in Clausen's work, our computation works for all
localizing invariants, not just $K$-theory.

\end{abstract}
\maketitle

\section{Introduction}

Let $\mathsf{LCA}$ denote the category of locally compact abelian (LCA)
groups, $\operatorname*{Cat}_{\infty}^{\operatorname*{ex}}$ the $\infty
$-category of small stable $\infty$-categories, $\mathsf{A}$ any stable
$\infty$-category. Recently, Dustin Clausen proved the following theorem:

\begin{theoremannounce}
[Clausen]For every localizing invariant $K:\operatorname*{Cat}_{\infty
}^{\operatorname*{ex}}\rightarrow\mathsf{A}$, there is a canonical fiber
sequence%
\[
K(\mathbb{Z})\longrightarrow K(\mathbb{R})\longrightarrow K(\mathsf{LCA}%
)\text{,}%
\]
where the first map is induced from the ring homomorphism $\mathbb{Z}%
\rightarrow\mathbb{R}$.
\end{theoremannounce}

This is \cite[Theorem 3.4]{clausen}. Clausen's work is based on earlier ideas
of Hoffmann and Spitzweck \cite{MR2329311}. As an example, we can take
non-connective $K$-theory for $K$ (the letter was chosen suggestively), which
has values in spectra, and obtain a computation of the $K$-theory of the
category $\mathsf{LCA}$.\medskip

We generalize Clausen's result as follows:

\begin{theorem}
\label{thm_intro_Main1}Let $F$ be a number field and $\mathcal{O}$ its ring of
integers. Let $\mathsf{LCA}_{\mathcal{O}}$ be the category of locally compact
$\mathcal{O}$-modules. For every localizing invariant $K:\operatorname*{Cat}%
_{\infty}^{\operatorname*{ex}}\rightarrow\mathsf{A}$, there is a canonical
fiber sequence%
\[
K(\mathcal{O})\longrightarrow K(\mathbb{R})^{r}\oplus K(\mathbb{C}%
)^{s}\longrightarrow K(\mathsf{LCA}_{\mathcal{O}})\text{,}%
\]
where $r$ is the number of real places and $s$ the number of complex places.
\end{theorem}

See Theorem \ref{thm_MainThm} below for details. For $\mathcal{O}=\mathbb{Z}$,
the category $\mathsf{LCA}_{\mathcal{O}}$ agrees with the category
$\mathsf{LCA}$, so we get the same statement as in Clausen's result. Beyond
this, the category $\mathsf{LCA}_{\mathcal{O}}$ appears to reflect some phenomena around the infinite primes in
a very natural way. This is quite remarkable since many such theories require
some manual handling for infinite primes.

The category $\mathsf{LCA}_{\mathcal{O}}$ is of glaring beauty:\ $\mathsf{LCA}%
_{\mathcal{O}}$ is a quasi-abelian exact category with duality. The Minkowski
embedding gives rise to an exact sequence%
\[
\mathcal{O}\hookrightarrow\bigoplus_{\sigma\in S}\mathbb{R}_{\sigma
}\twoheadrightarrow\mathbb{T}_{\mathcal{O}}%
\]
($S$ the set of infinite places, $\mathbb{R}_{\sigma}$ the codomain of
$\sigma$, and $\mathbb{T}_{\mathcal{O}}$ the torus quotient). In
$\mathsf{LCA}_{\mathcal{O}}$ this sequence is simultaneously, (1) an injective
resolution of $\mathcal{O}$, (2) a projective resolution of $\mathbb{T}%
_{\mathcal{O}}$, and (3) the Pontryagin dual of the sequence is isomorphic to
the sequence itself. So the Minkowski embedding is hardcoded in the
homological algebra of $\mathsf{LCA}_{\mathcal{O}}$.

Similarly for the Dirichlet embedding: For algebraic $K$-theory, the sequence
in Theorem \ref{thm_intro_Main1} in degree one yields%
\[
K_{1}(\mathcal{O})\longrightarrow(\mathbb{R}^{\times}\mathbb{)}^{r}%
\oplus(\mathbb{C}^{\times})^{s}\longrightarrow K_{1}(\mathsf{LCA}%
_{\mathcal{O}})\longrightarrow\operatorname{Cl}(\mathcal{O})\longrightarrow0\text{,}%
\]
which essentially identifies $K_{1}(\mathsf{LCA}_{\mathcal{O}})$ as
an extension of the class group with $\mathbb{R}_{>0}^{\times}\times($torus$)$, naturally containing the torus
arising from Dirichlet's Unit Theorem. The free real factor $\mathbb{R}%
_{>0}^{\times}$ corresponds to choosing a normalization for the Haar measure
the underlying LCA\ group.

These facts resemble regulator constructions. Indeed, making use of properties
of the Borel regulator, we get:

\begin{theorem}
Let $F$ be a number field, $\mathcal{O}$ its ring of integers, and $K$
(non-connective) algebraic $K$-theory. With rational coefficients, the long
exact sequence induced from the theorem above splits into short exact
sequences%
\[
0\rightarrow K_{n}(\mathcal{O})_{\mathbb{Q}}\rightarrow K_{n}(\mathbb{R}%
)_{\mathbb{Q}}^{r}\oplus K_{n}(\mathbb{C})_{\mathbb{Q}}^{s}\rightarrow
K_{n}(\mathsf{LCA}_{\mathcal{O}})_{\mathbb{Q}}\rightarrow0
\]
for all $n$.
\end{theorem}

See Theorem \ref{thm_BorelReg}. Our methods are quite different from
Clausen's in \cite{clausen}. While Clausen sets up a homotopy coherent cone
construction in the context of stable $\infty$-categories, this paper
exclusively deals with exact categories. Our method is based on techniques of
Schlichting (namely left and right $s$-filtering subcategories) in order to
show that certain quotient constructions do not only make sense on a derived
level, but admit concrete exact category models. The main technical tool is the following:

\begin{theorem}
Let $F$ be a number field, $\mathcal{O}$ its ring of integers. Then
\begin{enumerate}
\item The compactly generated $\mathcal{O}$-modules $\mathsf{LCA}_{\mathcal{O},cg}$ are left
\textit{s}-filtering in $\mathsf{LCA}_{\mathcal{O}}$.
\item The $\mathcal{O}$-modules $\mathsf{LCA}_{\mathcal{O},nss}$ without small subgroups are right
\textit{s}-filtering in $\mathsf{LCA}_{\mathcal{O}}$.
\end{enumerate}
\end{theorem}

These results essentially amount to
setting up a calculus of left or right fractions with certain additional
properties. They might be of independent interest. Thus, even in the case $\mathcal{O}=\mathbb{Z}$, our proof is
different from the one in \cite{clausen}.\medskip

Needless to say, this article is heavily inspired by Clausen's work.\medskip

\textit{Conventions: }We shall work a lot with exact categories. We follow the
nomenclature of B\"{u}hler's survey \cite{MR2606234}. In particular
\textquotedblleft$\hookrightarrow$\textquotedblright\ resp. \textquotedblleft%
$\twoheadrightarrow$\textquotedblright\ denote admissible monics resp.
admissible epics with respect to the exact structure. For stable $\infty
$-categories and localizing invariants, we employ the conventions of
\cite{MR3070515}.\medskip

\section{\label{sect_LocCompModules}The category of locally compact modules}

Let $\mathsf{LCA}$ be the category of locally compact abelian groups with
continuous group homomorphisms as its morphisms. Hoffmann and\ Spitzweck have
observed that $\mathsf{LCA}$ is a quasi-abelian category \cite{MR2329311}, and
thus an exact category. This can be generalized as follows:

\begin{definition}
Let $R$ be a commutative unital ring. Regard $R$ as equipped with the discrete
topology. Let $\mathsf{LCA}_{R}$ be the category of locally compact
$R$-modules with continuous $R$-module homomorphisms as morphisms.
\end{definition}

This means that an object in $\mathsf{LCA}_{R}$ is (1) an$\mathcal{\ }%
R$-module $M$, (2) that its additive group $(M,+)$ comes equipped with the
structure of an LCA group, and (3) the constraint that we demand each element
of $R$ to act on $(M,+)$ via a continuous endomorphism.

\begin{example}
We have $\mathsf{LCA}_{\mathbb{Z}}=\mathsf{LCA}$, i.e. this definition
encompasses the category of LCA groups as a special case. The full subcategory
of discrete $R$-modules in $\mathsf{LCA}_{R}$ is literally isomorphic to the
category of all $R$-modules.
\end{example}

\begin{proposition}
The category $\mathsf{LCA}_{R}$ is quasi-abelian.
\end{proposition}

\begin{proof}
The proof of Hoffmann and Spitzweck for $R=\mathbb{Z}$ carries over verbatim,
\cite[Prop. 1.2]{MR2329311}.
\end{proof}

In particular, $\mathsf{LCA}_{R}$ is an exact category in a natural way
\cite[Prop. 4.4]{MR2606234}. With respect to this structure, the admissible
monics are the closed injective $R$-module homomorphisms, while the admissible
epics are open surjective $R$-module homomorphisms. All kernels and cokernels
exist; in particular the category is idempotent complete so that arguments
like \cite[Cor. 7.7]{MR2606234} apply. The category is never abelian.

Recall that every morphism $f:A\rightarrow B$ in a quasi-abelian category
$\mathsf{C}$ has a canonical factorization%
\[
A\twoheadrightarrow\operatorname*{coim}\nolimits_{\mathsf{C}}(f)\overset
{r}{\rightarrow}\operatorname*{im}\nolimits_{\mathsf{C}}(f)\hookrightarrow
B\text{,}%
\]
where $\operatorname*{coim}\nolimits_{\mathsf{C}}(f):=A/\ker(f)$ is called the
\emph{coimage}. In an abelian category one has the luxury of $r$ being an
isomorphism. This is usually false for quasi-abelian categories and false for
$\mathsf{LCA}_{R}$ in particular.

\begin{example}
The inclusion $\mathbb{Q}\rightarrow\mathbb{R}$ is a monic in $\mathsf{LCA}$,
but not an admissible monic. The coimage is $\mathbb{Q}$, while the image is
$\mathbb{R}$. The Pontryagin dual is $\mathbb{R}\rightarrow\mathbb{A}%
_{\mathbb{Q}}/\mathbb{Q}$, where $\mathbb{A}_{\mathbb{Q}}$ denotes the
rational ad\`{e}les. This map is an epic, but not an admissible epic. The
quotient $\mathbb{A}_{\mathbb{Q}}/\mathbb{Q}$ is a compact connected
$\mathbb{Q}$-vector space.
\end{example}

\begin{example}
If $\mathbb{R}_{\delta}$ denotes the reals with the discrete topology,
$\mathbb{R}_{\delta}\rightarrow\mathbb{R}$ is a morphism in $\mathsf{LCA}$. It
is both monic and epic, yet not an isomorphism. It is neither an admissible
monic nor an admissible epic. Its Pontryagin dual is $\mathbb{R}\rightarrow
b\mathbb{R}$, the map underlying the Bohr compactification.
\end{example}

\begin{notation}
In the literature for topological groups the notation $\operatorname*{im}%
(f)$\ usually refers to the set-theoretic image. In the present article we run
into the trouble that this is not the same as the image object
$\operatorname*{im}\nolimits_{\mathsf{LCA}}(f)$ in the category $\mathsf{LCA}%
$. However, it feels more than awkward to write $\operatorname*{coim}%
\nolimits_{\mathsf{LCA}}(f)$, which would be the appropriate term from the
viewpoint of category theory. For this reason, we shall always write
$\operatorname*{im}\nolimits_{\mathsf{Set}}(f)$ and speak of the
\emph{set-theoretic image}, as a compromise to create a both
category-theoretically correct yet reader-friendly text.
\end{notation}

We return to $\mathsf{LCA}$. We denote the circle group by $\mathbb{T}%
:=\mathbb{R}/\mathbb{Z}$ and write $G^{\vee}:=\operatorname*{Hom}%
\nolimits_{cts}(G,\mathbb{T})$ for the Pontryagin dual, where the homomorphism
group is equipped with the compact-open topology. Pontryagin duality then
yields an exact anti-equivalence%
\[
(-)^{\vee}:\mathsf{LCA}^{op}\overset{\sim}{\longrightarrow}\mathsf{LCA}%
\text{.}%
\]
If $f:G\rightarrow H$ is injective, then $f^{\vee}:H^{\vee}\rightarrow
G^{\vee}$ has dense set-theoretic image; if $f$ has dense set-theoretic image,
$f^{\vee}$ is injective. If $f$ is an admissible monic, $f^{\vee}$ is an
admissible epic, and conversely.

This generalizes to $\mathsf{LCA}_{R}$ as follows: Although $\mathbb{T}$ does
not carry a natural $R$-module structure in any way and thus only makes sense
in $\mathsf{LCA}$, the dual $G^{\vee}$ of an $R$-module again carries an
$R$-module structure by%
\[
(a\cdot\chi)(m):=\chi(a\cdot m)\qquad\text{for}\qquad a\in R\text{, }\chi\in
G^{\vee}\text{, }m\in G\text{.}%
\]
This leads to a generalized form of Pontryagin duality:

\begin{theorem}
\label{thm_PontrDualForLCAO}There is an exact anti-equivalence of exact
categories%
\[
(-)^{\vee}:\mathsf{LCA}_{R}^{op}\overset{\sim}{\longrightarrow}\mathsf{LCA}%
_{R}\text{,}%
\]
reflexive in the sense that the natural morphism $\eta_{G}:G\overset{\sim
}{\rightarrow}G^{\vee\vee}$ is an isomorphism. On the level of the underlying
additive groups, this agrees with ordinary Pontryagin duality.\ In particular,
it sends compact $R$-modules to discrete $R$-modules, and conversely.
\end{theorem}

This extension of Pontryagin duality has been observed by a number of people
in various versions, e.g.,\ St\"{o}hr \cite{MR0262223} and Levin \cite[Theorem
1]{MR0310125}.

In fact, we can combine the exact structure with this duality:

\begin{proposition}
The category $\mathsf{LCA}_{R}$, equipped with Pontryagin duality, is an exact
category with duality (in the sense of \cite[Definition 2.1]{MR2600285}).
\end{proposition}

Of course this is precisely what everyone would expect, yet I know of no
literature studying the category $\mathsf{LCA}$ from the perspective of an
exact category with duality.

\begin{proof}
Standard.
\end{proof}

We return to $\mathsf{LCA}$:

\begin{definition}
Let $G$ be an LCA\ group. A subset $U\subseteq G$ is called \emph{symmetric}
if it is closed under taking inverses.

\begin{enumerate}
\item The group $G$ has \emph{no small subgroups} if there exists a symmetric
open neighbourhood $U\subset G$ of the neutral element such that $U $ contains
no non-trivial subgroup.

\item The group $G$ is \emph{compactly generated} if there exists a compact
symmetric neighbourhood $U\subset G$ of the neutral element such that
$G=\bigcup_{n\geq1}U^{n}$, where $U^{n}$ denotes the image of the product of
any $n$ elements in $U$.
\end{enumerate}

Let $\mathsf{LCA}_{nss}$ (resp. $\mathsf{LCA}_{cg}$) be the full subcategory
of groups having no small subgroups (resp. being compactly generated).
\end{definition}

The same qualifications make sense for objects in $\mathsf{LCA}_{R}$. We write
$\mathsf{LCA}_{R,nss}$ resp. $\mathsf{LCA}_{R,cg}$ for those objects whose
underlying additive groups have no small subgroups resp. are compactly generated.

Just as compact and discrete groups are exchanged under Pontryagin duality, so
are the above concepts:

\begin{proposition}
\label{prop_MoskowitzInterchange}The category $\mathsf{LCA}_{R,nss}$ (resp.
$\mathsf{LCA}_{R,cg}$) is extension-closed in $\mathsf{LCA}_{R}$. Moreover,
both are dual partners under Pontryagin duality, i.e. $\mathsf{LCA}_{R}%
^{op}\longrightarrow\mathsf{LCA}_{R}$ sends $\mathsf{LCA}_{R,nss}^{op}$ to
$\mathsf{LCA}_{R,cg}$ and conversely.
\end{proposition}

\begin{proof}
This is proven in Moskowitz for $R=\mathbb{Z}$, see \cite[Theorem
2.6]{MR0215016}. However, since these qualities only rely on the topology of
the underlying additive group, this is sufficient in our situation as well.
\end{proof}

\subsection{Number fields}

Let $F$ be a number field and $\mathcal{O}$ its ring of integers. Regard
$\mathcal{O}$ as equipped with the discrete topology. Then $\mathsf{LCA}%
_{\mathcal{O}}$ is the category of locally compact $\mathcal{O}$-modules with
continuous $\mathcal{O}$-module homomorphisms as morphisms.

Suppose%
\[
\sigma:F\longrightarrow\mathbb{R}\text{ (resp. }\mathbb{C}\text{)}%
\]
is a real or complex embedding of the number field (i.e. $\sigma$ is a ring homomorphism).

\begin{definition}
We define $\mathbb{R}_{\sigma}$ to be the $\mathcal{O}$-module whose
underlying additive group is $\mathbb{R}$ (resp. $\mathbb{C}$) and whose
$\mathcal{O}$-module structure is given by%
\[
a\cdot\beta:=\sigma(a)\beta\qquad\text{for}\qquad a\in\mathcal{O}\text{,
}\beta\in\mathbb{R}\text{ (resp. }\beta\in\mathbb{C}\text{).}%
\]

\end{definition}

Note that we write $\mathbb{R}_{\sigma}$ even if $\sigma$ comes from a
strictly complex embedding. This is convenient in order to avoid repeatedly
having to go through case distinctions for real and complex places. It is easy
to see that $\mathbb{R}_{\sigma}\in\mathsf{LCA}_{\mathcal{O}}$. We have
isomorphy%
\[
\mathbb{R}_{\sigma}^{\vee}\simeq\mathbb{R}_{\sigma}\text{,}%
\]
even though there is no canonical isomorphism.

We recall that an LCA group is called a \emph{vector group} if it admits an
isomorphism to $\mathbb{R}^{n}$ for some $n\geq0$. We shall call $G$ in
$\mathsf{LCA}_{\mathcal{O}}$ a \emph{vector }$\mathcal{O}$-\emph{module} if
its underlying additive group is a vector group. The following result goes
back to Levin \cite{MR0310125}.

\begin{proposition}
\label{prop_ClassifVectorOModules}If $G\in\mathsf{LCA}_{\mathcal{O}}$ is a
vector $\mathcal{O}$-module, $G$ is a finite direct sum $G\simeq
\bigoplus_{\sigma\in I}\mathbb{R}_{\sigma}$, where $I$ is a finite list of
real and complex places.
\end{proposition}

We write \textquotedblleft list\textquotedblright\ with the meaning that
repetitions are allowed.

\begin{proof}
(follows Levin, Part 4 of the proof of \cite[Prop. 1]{MR0310125}) We give a
quick sketch, as this argument plays a crucial r\^{o}le: By a standard
argument, every $\mathsf{LCA}$ endomorphism of a vector group is $\mathbb{R}%
$-linear (\textit{Sketch: }it is an abelian group homomorphism; then since $G$
is uniquely divisible, it must be $\mathbb{Q}$-linear and by continuity must
be $\mathbb{R}$-linear). Thus, $\mathcal{O}$ acts by $\mathbb{R}$-vector space
endomorphisms on $G$. Let $\lambda\in F$ be a primitive element of the number
field so that $F=\mathbb{Q}(\lambda)$. Without loss of generality, we may
demand $\lambda\in\mathcal{O}$; let $f$ be its minimal polynomial. Then the
$\mathcal{O}$-module structure is entirely determined by $\mathbb{R}%
$-linearity and the action of $\lambda$, so $G$ canonically is an
$\mathbb{R}[T]/(f)$-module, where $T$ acts as $T(m):=\lambda\cdot m$. So the
$\mathcal{O}$-module structure uniquely defines an $\mathbb{R}[T]/(f)$-module
structure. Conversely, every $\mathbb{R}[T]/(f)$-module is an $F$-vector space
and thus an $\mathcal{O}$-module. We see that a direct sum decomposition of
$G$ as an $\mathbb{R}[T]/(f)$-module agrees with a direct sum decomposition as
$\mathcal{O}$-modules; and summands are simple in one sense iff they are
simple in the other. Polynomials over $\mathbb{R}$ factor in linear and
quadratic factors, and these correspond precisely to the real and complex
embeddings; e.g. in the first case the simple summands have the shape
$\mathbb{R}[T]/(T-\sigma(\lambda))$; but this is just $\mathbb{R}_{\sigma}$;
and correspondingly for complex embeddings.
\end{proof}

\begin{corollary}
\label{cor_LCAOR}Let $\mathsf{LCA}_{\mathcal{O},\mathbb{R}}$ be the full
subcategory of $\mathsf{LCA}_{\mathcal{O}}$ whose objects are vector
$\mathcal{O}$-modules. This is an extension-closed subcategory. For every
localizing invariant $K:\operatorname*{Cat}_{\infty}^{\operatorname*{ex}%
}\rightarrow\mathsf{A}$, there are canonical equivalences%
\[
K(\mathsf{LCA}_{\mathcal{O},\mathbb{R}})\cong K(\mathbb{R})^{r}\oplus
K(\mathbb{C})^{s}\text{,}%
\]
where $r$ is the number of real places of $F$, and $s$ the number of complex places.
\end{corollary}

\begin{proof}
As in the previous proof, we may equivalently classify finitely generated
$\mathbb{R}[T]/(f)$-modules, where $f$ is a minimal polynomial of a primitive
element of $F$ over the rationals. Thus, we have exact equivalences of exact
categories%
\[
\mathsf{LCA}_{\mathcal{O},\mathbb{R}}\overset{\sim}{\longrightarrow
}\mathsf{Mod}_{fg}(\mathbb{R}[T]/(f))\overset{\sim}{\longrightarrow}\prod
_{i}\mathsf{Mod}_{fg}(\mathbb{R}[T]/(f_{i}))\text{,}%
\]
where $f_{i}$ are the irreducible factors so that $\mathbb{R}[T]/(f)\cong
\prod_{i}\mathbb{R}[T]/(f_{i})\cong\mathbb{R}^{r}\times\mathbb{C}^{s}$ holds
as an isomorphism of rings.
\end{proof}

\begin{definition}
We introduce some further categories:

\begin{enumerate}
\item Let $\mathsf{LCA}_{\mathcal{O},\mathbb{R}C}$ be the full subcategory of
$\mathsf{LCA}_{\mathcal{O}}$ of objects which admit an isomorphism to $V\oplus
C$ for some vector $\mathcal{O}$-module $V$ and compact $\mathcal{O}$-module
$C$.

\item Let $\mathsf{LCA}_{\mathcal{O},\mathbb{R}D}$ be the full subcategory of
$\mathsf{LCA}_{\mathcal{O}}$ of objects which admit an isomorphism to $V\oplus
D$ for some vector $\mathcal{O}$-module $V$ and discrete $\mathcal{O} $-module
$D$.
\end{enumerate}
\end{definition}

We have the full subcategory inclusions%
\[
\mathsf{LCA}_{\mathcal{O},\mathbb{R}C}\subset\mathsf{LCA}_{\mathcal{O}%
,cg}\qquad\text{and}\qquad\mathsf{LCA}_{\mathcal{O},\mathbb{R}D}%
\subset\mathsf{LCA}_{\mathcal{O},nss}\text{.}%
\]

\begin{lemma}
[{\cite[Lemma 2]{MR0310125}}]\label{lemma_VectorGroupDiscreteInLCALiftsToLCAO}%
Suppose the underlying additive group of an object $G\in\mathsf{LCA}%
_{\mathcal{O}}$ admits a direct sum splitting%
\[
G\simeq H\oplus D\qquad\text{in}\qquad\mathsf{LCA}%
\]
with $H\simeq\mathbb{R}^{n}\oplus\mathbb{T}^{m}$ for some $n,m\geq0$ finite,
and $D$ discrete. Then

\begin{enumerate}
\item $H$ and $D$ are closed under $\mathcal{O}$-multiplication, so that
$H,D\in\mathsf{LCA}_{\mathcal{O}}$ canonically, and

\item the given splitting comes from a direct sum splitting in $\mathsf{LCA}%
_{\mathcal{O}}$ and $H$ is a vector $\mathcal{O}$-module.
\end{enumerate}
\end{lemma}

\begin{proof}
Note that $H$ is the connected component of zero in $G$. As multiplication by
an element $a\in\mathcal{O}$ acts continuously, the image of a connected set
is connected, so it must lie in $H$ again. Thus, the subgroup $H$ is closed
under the action of $\mathcal{O}$ and thus $H\subseteq G$ is a topological
$\mathcal{O}$-submodule. Next, $H$ is divisible as a $\mathbb{Z}$-module, and
thus divisible as an algebraic $\mathcal{O}$-module. Since $\mathcal{O}$ is a
Dedekind domain, this implies that $H$ is an injective algebraic $\mathcal{O}%
$-module, i.e. it is a direct summand of $G$ on the level of the underlying
$\mathcal{O}$-module. The complement $D$ thus is also closed under the action
of $\mathcal{O}$. Since $D$ is discrete, this action is tautologically also
continuous. Thus, $H,D\in\mathsf{LCA}_{\mathcal{O}}$ and the direct sum
decomposition holds in $\mathsf{LCA}_{\mathcal{O}}$.
\end{proof}

Every ideal $J$ of $\mathcal{O}$ can be equipped with the discrete topology
and thus defines an object of $\mathsf{LCA}_{\mathcal{O}}$ in a natural way.
We shall tacitly write $J\in\mathsf{LCA}_{\mathcal{O}}$. Recall that under the
Minkowski embedding%
\[
M:\mathcal{O}\longrightarrow\bigoplus_{\sigma\in S}\mathbb{R}_{\sigma}\text{,}%
\]
where $\sigma$ runs through all real and complex places, a non-trivial ideal
(or more generally: fractional ideal) $J$ is sent to a full rank lattice in a
vector group. Note that $M$ is also an $\mathcal{O}$-module homomorphism;
continuous since $\mathcal{O}$ is discrete. Thus, we get an exact sequence in
$\mathsf{LCA}_{\mathcal{O}}$,%
\begin{equation}
J\hookrightarrow\bigoplus_{\sigma\in S}\mathbb{R}_{\sigma}\twoheadrightarrow
\mathbb{T}_{J}\text{,} \label{lMinkSeqForIdeals}%
\end{equation}
where $\mathbb{T}_{J}$ is our notation for the quotient. We call it the
\emph{Minkowski sequence} of $J$. Of course, the underlying LCA group of
$\mathbb{T}_{J}$ is a torus, justifying the notation. Note further that under
Pontryagin duality this sequence is self-dual, in that%
\[
J^{\vee}\twoheadleftarrow\bigoplus_{\sigma\in S}\mathbb{R}_{\sigma}^{\vee
}\hookleftarrow\mathbb{T}_{J}^{\vee}%
\]
is non-canonically isomorphic to the previous exact sequence.

Next, we classify the admissible subobjects of vector $\mathcal{O}$-modules.

\begin{proposition}
\label{Prop_ClosedSubobjectsOfVectorModules}Suppose $G\in\mathsf{LCA}%
_{\mathcal{O}}$ is a vector $\mathcal{O}$-module and $\alpha:H\hookrightarrow
G$ an admissible monic. Then $\alpha$ can be written as a finite $\mathcal{O}%
$-linear direct sum of morphisms

\begin{enumerate}
\item $0\hookrightarrow\mathbb{R}_{\sigma}$, for $\sigma$ a real or complex place,

\item $\mathbb{R}_{\sigma}\overset{1}{\rightarrow}\mathbb{R}_{\sigma}$ for
$\sigma$ a real or complex place,

\item $J\hookrightarrow\bigoplus_{\sigma}\mathbb{R}_{\sigma}$ the Minkowski
embedding of Equation \ref{lMinkSeqForIdeals}.
\end{enumerate}
\end{proposition}

\begin{proof}
We work on the level of the underlying additive LCA group at first: Then by
our assumptions $G\simeq\mathbb{R}^{n}$ for some $n\geq0$ in $\mathsf{LCA}$
and $H$ is a closed subgroup. By the classification of such closed subgroups,
\cite[Theorem 6]{MR0442141}, there exists an $\mathbb{R}$-vector space basis
$x_{1},\ldots,x_{n}$ of $\mathbb{R}^{n}$ and some integers $0\leq p\leq q\leq
n$ such that%
\begin{equation}
H=\mathbb{Z}\left\langle x_{1},\ldots,x_{p}\right\rangle \oplus\mathbb{R}%
\left\langle x_{p+1},\ldots,x_{q}\right\rangle \subseteq G\text{,}
\label{lcsis5}%
\end{equation}
where the $\mathbb{Z}$-part sits in $\mathbb{R}\left\langle x_{1},\ldots
,x_{p}\right\rangle $ as a discrete full rank lattice. So far, this is only a
direct sum in $\mathsf{LCA}$, however by Lemma
\ref{lemma_VectorGroupDiscreteInLCALiftsToLCAO} it must come from a direct sum
decomposition in $\mathsf{LCA}_{\mathcal{O}}$. Hence, the second summand must
be a vector $\mathcal{O}$-module, and thus is described by Proposition
\ref{prop_ClassifVectorOModules}, and the former a discrete $\mathcal{O}%
$-module, necessarily finitely generated as it is module-finite over
$\mathbb{Z}$. Hence, in $\mathsf{LCA}_{\mathcal{O}}$ the admissible monic
$\alpha:H\hookrightarrow G$ has the shape%
\[
\alpha:M\oplus\coprod_{\sigma\in I_{1}}\mathbb{R}_{\sigma}\hookrightarrow
\coprod_{\sigma\in I_{2}}\mathbb{R}_{\sigma}\text{,}%
\]
where $I_{1},I_{2}$ are finite lists of real and complex places, and $M$ is a
finitely generated $\mathcal{O}$-module, discretely embedded into the
right-hand side as a lattice. The map originating from the summand $M$ factors
over $M\rightarrow M\otimes_{\mathcal{O}}F$ since the vector $\mathcal{O}%
$-modules on the right are $F$-vector spaces. Hence, if $f\in\mathbb{Q}[T]$
denotes the minimal polynomial of a primitive element $\lambda\in F$, then
this $\mathcal{O}$-module homomorphism is of the shape%
\[
M\rightarrow M\otimes_{\mathcal{O}}\mathbb{Q}[T]/(f)\overset{\tilde{\alpha}%
}{\rightarrow}\coprod_{\sigma\in I_{2}}\mathbb{R}[T]/(f_{\sigma})\text{,}%
\]
where we write $f_{\sigma}$ for the irreducible factor of $f$ over
$\mathbb{R}$ corresponding to the place $\sigma$ (so quadratic resp. linear
for real resp. complex places $\sigma\in I_{2}$); and $\tilde{\alpha}$ is an
$F$-vector space homomorphism. As $\mathcal{O}$ is a Dedekind domain, and
since $M$ is finitely generated torsion-free, $M$ must be projective and thus
$M$ is isomorphic to a finite direct sum of ideals $J$ of $\mathcal{O}$. Thus,
in order to prove our claim, it suffices to assume $M=J$ for $J$ such an
ideal. Now proceed as follows: If the list $I_{2}$ in the morphism%
\[
J\longrightarrow\coprod_{\sigma\in I_{2}}\mathbb{R}_{\sigma}%
\]
has duplicates, then the image factors over the subobject given by $I_{2}$,
but with all duplicates removed (this arises as a subobject through the
diagonal map of the single copy, embedding diagonally into the summands of all
replica of this single copy appearing in $I_{2}$). On the other hand, each
real and complex place must appear in $I_{2}$ at least once: The discussion
around Equation \ref{lMinkSeqForIdeals} implies that $J\rightarrow
J\otimes_{\mathcal{O}}F$ is the embedding as a fractional ideal and since this
is a \textit{full rank} embedding, the standard results imply that any
projection to a proper subset of embeddings cannot have discrete image. This
means that its image in $G$ can only be discrete if all real and complex
places $\sigma$ appear on the right at least once, letting $J\rightarrow G$
factor over a Minkowski embedding as in Equation \ref{lMinkSeqForIdeals}.
Thus, the map $\alpha$ is a direct sum of maps of the shape $\mathbb{R}%
_{\sigma}\overset{1}{\rightarrow}\mathbb{R}_{\sigma}$ and Minkowski embeddings
$J\hookrightarrow%
{\textstyle\bigoplus_{\sigma\in I^{\prime}}}
\mathbb{R}_{\sigma}$ as in Equation \ref{lMinkSeqForIdeals} (where $I^{\prime
}$ features each real and complex place exactly once).
\end{proof}

Next, we shall use the following structure result due to Levin:

\begin{proposition}
[{\cite[Prop. 1]{MR0310125}}]\label{prop_MoskClassifNSSandCGGroups}We have the
following structure results for objects in $\mathsf{LCA}_{\mathcal{O}}$:

\begin{enumerate}
\item Every object in $\mathsf{LCA}_{\mathcal{O},nss}$ is isomorphic to
$\bigoplus_{\sigma\in I}\mathbb{R}_{\sigma}\oplus\bigoplus_{J\in\mathcal{I}%
}\mathbb{T}_{J}\oplus D$ with $D$ discrete $\mathcal{O}$-module, $\mathcal{I}$
a finite list of ideals in $\mathcal{O}$, and $I$ a finite list of real and
complex places.

\item Every object in $\mathsf{LCA}_{\mathcal{O},cg}$ is isomorphic to
$\bigoplus_{\sigma\in I}\mathbb{R}_{\sigma}\oplus\bigoplus_{J\in\mathcal{I}%
}J\oplus C$ with $C$ a compact $\mathcal{O}$-module, $\mathcal{I}$ a finite
list of ideals in $\mathcal{O}$, and $I$ a finite list of real and complex places.
\end{enumerate}

Here an ideal $J$ is understood to carry the discrete topology, and
$\mathbb{T}_{J}$ is as explained above.
\end{proposition}

\begin{proof}
This is proven by \cite[Prop. 1]{MR0310125} for $\mathsf{LCA}_{\mathcal{O}%
,cg}$. We follow his ideas, but give a shorter treatment by allowing ourselves
to use the classical structure theorems for $\mathsf{LCA}_{cg}$ as in
\cite[Theorem 2.4]{MR0215016} in full. (Step 1) Suppose $G\in\mathsf{LCA}%
_{\mathcal{O},nss}$. Then the underlying additive group has no small
subgroups, so by \cite[Theorem 2.4]{MR0215016} there is an isomorphism of
LCA\ groups%
\begin{equation}
G\simeq\mathbb{R}^{n}\oplus\mathbb{T}^{m}\oplus D\qquad\text{(in }%
\mathsf{LCA}\text{)} \label{lcsis1}%
\end{equation}
for some $n,m\geq0$ and $D$ discrete. By Lemma
\ref{lemma_VectorGroupDiscreteInLCALiftsToLCAO}, we can promote the second
direct sum in Equation \ref{lcsis1}, underlined in%
\[
G\simeq\lbrack\mathbb{R}^{n}\oplus\mathbb{T}^{m}]\underline{\oplus}%
D\qquad\text{(in }\mathsf{LCA}_{\mathcal{O}}\text{),}%
\]
to a direct sum decomposition in the category $\mathsf{LCA}_{\mathcal{O}}$.
(Step 2) Now consider the object $[\mathbb{R}^{n}\oplus\mathbb{T}^{m}]$ in
$\mathsf{LCA}_{\mathcal{O}}$. On the level of the underlying LCA group, its
Pontryagin dual, just on the level of $\mathsf{LCA}$, satisifes $[\mathbb{R}%
^{n}\oplus\mathbb{T}^{m}]^{\vee}\simeq\mathbb{R}^{n}\oplus\mathbb{Z}^{m}$ (in
$\mathsf{LCA}$). Run the same argument as before, Lemma
\ref{lemma_VectorGroupDiscreteInLCALiftsToLCAO} and so we can promote this
isomorphism to%
\[
\lbrack\mathbb{R}^{n}\oplus\mathbb{T}^{m}]^{\vee}\simeq\mathbb{R}%
^{n}\underline{\oplus}\mathbb{Z}^{m}\qquad\text{(in }\mathsf{LCA}%
_{\mathcal{O}}\text{),}%
\]
now again with the underlined direct sum on the level of $\mathsf{LCA}%
_{\mathcal{O}}$. Take Pontryagin duals again, this time in $\mathsf{LCA}%
_{\mathcal{O}}$, to get an isomorphism $[\mathbb{R}^{n}\oplus\mathbb{T}%
^{m}]^{\vee\vee}\cong\mathbb{R}^{n}\oplus\mathbb{T}^{m}$ in $\mathsf{LCA}%
_{\mathcal{O}}$. Combining these isomorphisms and using reflexivity of double
dualization, we have managed to promote the isomorphism in Equation
\ref{lcsis1} from a direct sum in $\mathsf{LCA}$ to one in $\mathsf{LCA}%
_{\mathcal{O}}$. (Step 3) If instead $G\in\mathsf{LCA}_{\mathcal{O},cg}$, use
Proposition \ref{prop_MoskowitzInterchange} to get $G^{\vee}\in\mathsf{LCA}%
_{\mathcal{O},nss}$, use the above structure theorem there and dualize again.
\end{proof}

\begin{lemma}
\label{lemma_ExistClopenCGSubmodule}Let $G\in\mathsf{LCA}_{\mathcal{O}}$ be
given. Then there exists a clopen $\mathcal{O}$-submodule $H$ which is
compactly generated.
\end{lemma}

\begin{proof}
If $G=0$ the claim is trivial. Thus, assume $G\neq0$. Let $g\in G$ be a
non-zero element. As $G$ is locally compact, we can find a compact
neighbourhood $U$ of $g$. As $\mathcal{O}$ is finite over the integers, pick
module generators so that $\mathcal{O}=\mathbb{Z}\left\langle \beta_{1}%
,\ldots,\beta_{n}\right\rangle $. Define%
\[
U^{\prime}:=\bigcup_{i=1}^{n}(\beta_{i}U\cup-\beta_{i}U)\text{.}%
\]
As a finite union of $2n$ compact sets, $U^{\prime}$ is again compact, and
clearly symmetric. Define $H:=\bigcup_{n\geq1}U^{\prime n}$. We conclude that
$H$ is compactly generated and clopen (cf. \cite[Prop. 8]{MR0442141}). We
claim that it is also an $\mathcal{O}$-submodule: For every $h\in H$, every
$\left(  \sum_{i=1}^{n}m_{i}\beta_{i}\right)  h$ with $m_{i}\in\mathbb{Z}$ can
be written as a finite sum of terms of the shape $\pm\beta_{i}h$ by writing
the $m_{i}$ as repeated sums of $\pm1$.
\end{proof}

\begin{theorem}
[{Levin \cite[Theorem 2]{MR0310125}}]\label{thm_StructOfLCA}Suppose $G$ is a
locally compact $\mathcal{O}$-module. Then there exists some vector
$\mathcal{O}$-module $\bigoplus_{\sigma\in I}\mathbb{R}_{\sigma}$ (for some
finite list $I$ of real and complex places) and a compact $\mathcal{O}%
$-submodule $C$ such that%
\[
\bigoplus_{\sigma\in I}\mathbb{R}_{\sigma}\oplus C\hookrightarrow
G\twoheadrightarrow D
\]
is an exact sequence in $\mathsf{LCA}_{\mathcal{O}}$ and $D$ a discrete
$\mathcal{O}$-module.
\end{theorem}

For $\mathcal{O}=\mathbb{Z}$, this becomes the classical Principal Structure
Theorem of LCA groups due to Pontryagin and van\ Kampen.

\begin{proof}
We give a very short proof: By Lemma \ref{lemma_ExistClopenCGSubmodule} we can
find an exact sequence $H^{\prime}\hookrightarrow G\twoheadrightarrow
D^{\prime}$ in $\mathsf{LCA}_{\mathcal{O}}$ with $H^{\prime}$ compactly
generated and $D^{\prime}$ discrete since $H$ is clopen \cite[Prop.
14]{MR0442141}. Thus, $H^{\prime}\simeq\bigoplus_{\sigma\in I}\mathbb{R}%
_{\sigma}\oplus\bigoplus_{J\in\mathcal{I}}J\oplus C$ by Proposition
\ref{prop_MoskClassifNSSandCGGroups}, so since each $J$ is discrete,
$\bigoplus_{\sigma\in I}\mathbb{R}_{\sigma}\oplus C$ is still an open and thus
clopen submodule of $G$.
\end{proof}

Next, we need to recall the structure of injective and projective objects in
$\mathsf{LCA}_{\mathcal{O}}$. For $\mathsf{LCA}$ these results are due to
Moskowitz \cite{MR0215016}, for $\mathsf{LCA}_{\mathcal{O}}$ due to Kryuchkov.

\begin{theorem}
[Kryuchkov, \cite{MR1620000}]\label{thm_InjectivesInLCAO}For an object
$G\in\mathsf{LCA}_{\mathcal{O}}$ the following are equivalent:

\begin{enumerate}
\item $G$ is an injective object.

\item $G\simeq\bigoplus_{\sigma\in I}\mathbb{R}_{\sigma}\oplus\prod
_{J\in\mathcal{I}}\mathbb{T}_{J}$, where $I$ is a finite list of real and
complex places and $\mathcal{I}$ a (possibly infinite) list of ideals in
$\mathcal{O}$.

\item The additive group of $G$ is an injective object in $\mathsf{LCA}$ and
the underlying algebraic $\mathcal{O}$-module is an injective object in
$\mathcal{O}$-modules.
\end{enumerate}
\end{theorem}

For the proof we refer to Kryuchkov \cite{MR1620000} or the Appendix
\S \ref{sect_AppendixInjectives}.

\begin{theorem}
[Kryuchkov, \cite{MR1620000}]\label{thm_ProjectivesInLCAO}For an object
$G\in\mathsf{LCA}_{\mathcal{O}}$ the following are equivalent:

\begin{enumerate}
\item $G$ is a projective object.

\item $G\simeq\bigoplus_{\sigma\in I}\mathbb{R}_{\sigma}\oplus\prod
_{J\in\mathcal{I}}J$, where $I$ is a finite list of real and complex places
and $\mathcal{I}$ a (possibly infinite) list of ideals in $\mathcal{O} $.

\item The additive group of $G$ is a projective object in $\mathsf{LCA}$ and
the underlying algebraic $\mathcal{O}$-module is a projective object in
$\mathcal{O}$-modules.
\end{enumerate}
\end{theorem}

\begin{proof}
Take Pontryagin duals, use Theorem \ref{thm_InjectivesInLCAO}, and dualize again.
\end{proof}

\begin{corollary}
An object $G\in\mathsf{LCA}_{\mathcal{O}}$ is simultaneously injective and
projective if and only if it is isomorphic to $\bigoplus_{\sigma\in
I}\mathbb{R}_{\sigma}$ for a finite list $I$ of real and complex places.
\end{corollary}

\begin{remark}
\label{rmk_SplitOffR}As each $\mathbb{R}_{\sigma}$ is an injective object in
$\mathsf{LCA}_{\mathcal{O}}$, they always split off as a direct summand. Thus,
instead of the presentation of Theorem \ref{thm_StructOfLCA}, we may write any
object $G\in\mathsf{LCA}_{\mathcal{O}}$ as $G\simeq\bigoplus_{\sigma\in
I}\mathbb{R}_{\sigma}\oplus\tilde{G}$ for some $\tilde{G}\in\mathsf{LCA}%
_{\mathcal{O}}$ with $H\hookrightarrow G\twoheadrightarrow D$ exact, $H$
compact and $D$ discrete.
\end{remark}

\begin{theorem}
[Kryuchkov, \cite{MR1620000}]\label{thm_ResolvableObjects}We have the following:

\begin{enumerate}
\item An object $G\in\mathsf{LCA}_{\mathcal{O}}$ has an injective resolution
if and only if it lies in $\mathsf{LCA}_{\mathcal{O},cg}$. In this case, it
admits a length $2$ injective resolution.

\item An object $G\in\mathsf{LCA}_{\mathcal{O}}$ has a projective resolution
if and only if it lies in $\mathsf{LCA}_{\mathcal{O},nss}$. In this case, it
admits a length $2$ projective resolution.
\end{enumerate}
\end{theorem}

\begin{proof}
(1) Suppose $G$ has an injective resolution $G\rightarrow I^{0}\rightarrow
I^{1}\rightarrow\cdots$. Then $G\hookrightarrow I^{0}$ is an admissible monic,
but $I^{0}\simeq\bigoplus_{\sigma\in I}\mathbb{R}_{\sigma}\oplus\prod
_{J\in\mathcal{I}}\mathbb{T}_{J}$ by the characterization of injectives,
Theorem \ref{thm_InjectivesInLCAO}. As the tori $\mathbb{T}_{J}$ are compact
and thus $\prod_{J\in\mathcal{I}}\mathbb{T}_{J}$ is compact by Tychonoff's
Theorem, and then Proposition \ref{prop_MoskClassifNSSandCGGroups} implies
that $I^{0}$ is compactly generated. However, closed subgroups of compactly
generated LCA groups are themselves compactly generated, \cite[Theorem
2.6]{MR0215016}, i.e. $G\in\mathsf{LCA}_{\mathcal{O},cg}$. Conversely, suppose
$G\in\mathsf{LCA}_{\mathcal{O},cg}$. By Proposition
\ref{prop_MoskClassifNSSandCGGroups} there exists an isomorphism%
\begin{equation}
G\simeq\bigoplus_{\sigma\in I}\mathbb{R}_{\sigma}\oplus\bigoplus
_{J\in\mathcal{I}}J\oplus C \label{lcis7}%
\end{equation}
in $\mathsf{LCA}_{\mathcal{O}}$ with $C$ a compact $\mathcal{O}$-module,
$\mathcal{I}$ a finite list of ideals in $\mathcal{O}$, and $I$ a finite list
of real and complex places.\newline(Part A) The dual $C^{\vee}$ is a discrete
$\mathcal{O}$-module. As $\mathcal{O}$ is a Dedekind domain, $C^{\vee}$ has an
algebraic $\mathcal{O}$-module resolution of length $2$ by projective
algebraic $\mathcal{O}$-modules, $P_{1}\hookrightarrow P_{0}\twoheadrightarrow
C^{\vee}$. This becomes an exact sequence in $\mathsf{LCA}_{\mathcal{O}}$
tautologically since the topology is discrete, and indeed a projective
resolution by Theorem \ref{thm_ProjectivesInLCAO} since the $P_{i}$ are
algebraically projective and the underlying additive group is isomorphic to
$\mathbb{Z}^{n} $ for a suitable $n$ and thus projective in $\mathsf{LCA}$,
\cite[Theorem 3.3]{MR0215016}. Dualize again, giving a length $2$ injective
resolution of $C$.\newline(Part B) For each ideal $J$ we have the Minkowski
sequence%
\[
J\hookrightarrow\bigoplus_{\sigma}\mathbb{R}_{\sigma}\twoheadrightarrow
\mathbb{T}_{J}\text{,}%
\]
where $\sigma$ runs through all real and complex places once, and both types
of terms $\mathbb{R}_{\sigma}$ resp. $\mathbb{T}_{J}$ are injective by Theorem
\ref{thm_InjectivesInLCAO}. Thus, this is a length $2$ injective resolution of
$J$.\newline(Part C) By Theorem \ref{thm_InjectivesInLCAO} the objects
$\mathbb{R}_{\sigma}$ are already injectives.\newline Now take the direct sum
of the exact sequences of Parts A, B, C for all summands appearing in Equation
\ref{lcis7}. As a finite direct sum of injectives, this remains an injective
resolution, now of $G$.\newline(2) The discussion of the projective resolution
follows from dualizing everything.
\end{proof}

In particular, $\mathsf{LCA}_{\mathcal{O}}$ neither has enough injectives nor
enough projectives, irrespective of the number field.

\section{Setting up a calculus of fractions}

We now establish some factorization properties which are important to settle
the issues mentioned in the introduction. In case they hold, they ensure the
existence of a quotient exact category.

These results form the technical core of the paper.

\begin{definition}
Let $\mathsf{D}$ be an exact category.

\begin{enumerate}
\item A full subcategory $\mathsf{C}\subset\mathsf{D}$ is called \emph{left
special} if for every $Z\in\mathsf{C}$ and every admissible epic
$G\twoheadrightarrow Z$ in the category $\mathsf{D}$ there exists a
commutative diagram%
\[%
%%
%X \ar@{^{(}->}[r] \ar[d] & Y \ar@{->>}[r] \ar[d] & Z \ar[d]^{\mathbf{1}_Z} \\
%F \ar@{^{(}->}[r] & G \ar@{->>}[r] & Z,
%}}}%
%%
\xymatrix{
X \ar@{^{(}->}[r] \ar[d] & Y \ar@{->>}[r] \ar[d] & Z \ar[d]^{\mathbf{1}_Z} \\
F \ar@{^{(}->}[r] & G \ar@{->>}[r] & Z,
}%
\]
where the rows are exact, and the objects in the top row all lie in
$\mathsf{C}$.

\item A full subcategory $\mathsf{C}\subset\mathsf{D}$ is called \emph{left
filtering} if every morphism $G\rightarrow X$ in $\mathsf{D}$ such that
$G\in\mathsf{C}$ admits a factorization%
\[%
%%
%G \ar[r] \ar[rd] & X \\
%& Z, \ar@{^{(}->}[u]
%}}}%
%%
\xymatrix{
G \ar[r] \ar[rd] & X \\
& Z, \ar@{^{(}->}[u]
}%
\]
where $Z\in\mathsf{C}$.
\end{enumerate}

We call $\mathsf{C}\subset\mathsf{D}$ right special (resp. right filtering) if
$\mathsf{C}^{op}\subset\mathsf{D}^{op}$ is left special (resp. left special).
We call $\mathsf{C}\subset\mathsf{D}$ \emph{left }$s$\emph{-filtering}, if it
is left filtering and left special, resp. right $s$-filtering if
$\mathsf{C}^{op}\subset\mathsf{D}^{op}$ is left $s$-filtering.
\end{definition}

We have given these definitions as in \cite[\S 2.2]{MR3510209}. These
formulations are due to B\"{u}hler and are equivalent to the original concepts
as introduced by Schlichting in \cite{MR2079996}. The equivalence of these
formulations is proven in \cite[Appendix A]{MR3510209}.\medskip

Some very natural candidates for forming quotients do actually not satisfy
Schlichting's conditions. This illustrates their subtle r\^{o}le:

\begin{example}
\label{example_1}The full subcategory $\mathsf{LCA}_{C}$ is \underline{not}
left (or right) special in $\mathsf{LCA}$. If it were left special, there
would exist a compact group $C$ as in the left diagram below:%
\begin{equation}%
%%
%& C \ar@{->>}[r] \ar[d] & \mathbb{T} \ar@{=}[d] \\
%\mathbb{Z} \ar@{^{(}->}[r] & \mathbb{R} \ar@{->>}[r] & \mathbb{T}
%}}}%
%%
\xymatrix{
& C \ar@{->>}[r] \ar[d] & \mathbb{T} \ar@{=}[d] \\
\mathbb{Z} \ar@{^{(}->}[r] & \mathbb{R} \ar@{->>}[r] & \mathbb{T}
}%
\qquad\qquad%
%%
%\mathbb{Z}_p \ar@{^{(}->}[r] \ar@{=}[d] & \mathbb{Q}_p \ar[d] \ar@
%{->>}[r] & {\mathbb{Q}_p}/{\mathbb{Z}_p} \\
%\mathbb{Z}_p \ar@{^{(}->}[r] & G
%}} }%
%%
\xymatrix{
\mathbb{Z}_p \ar@{^{(}->}[r] \ar@{=}[d] & \mathbb{Q}_p \ar[d] \ar@
{->>}[r] & {\mathbb{Q}_p}/{\mathbb{Z}_p} \\
\mathbb{Z}_p \ar@{^{(}->}[r] & G
}
\label{DDE1}%
\end{equation}
However, the image of $C$ in $\mathbb{R}$ necessarily is trivial, giving a
contradiction. If it were right special, there would exist a compact group $G
$ such that the diagram on the right commutes. The kernel of $\mathbb{Q}%
_{p}\rightarrow G$ is a closed subgroup and thus can only be $p^{i}%
\mathbb{Z}_{p}$ for some $i\in\mathbb{Z}$ (it cannot be trivial:\ If it were,
$\mathbb{Q}_{p}$ would be a closed subgroup of a compact group and thus itself
compact). However, these subgroups are even clopen. Thus, we get a
factorization $\mathbb{Q}_{p}\twoheadrightarrow\mathbb{Q}/p^{i}\mathbb{Z}%
_{p}\rightarrow G$ with $\mathbb{Q}/p^{i}\mathbb{Z}_{p}$ is discrete. Hence,
the image of the compact group $\mathbb{Z}_{p}$ in $\mathbb{Q}/p^{i}%
\mathbb{Z}_{p}$ is finite, and thus its image in $G$ is finite. This
contradicts the injectivity of the map in the bottom row.
\end{example}

\begin{example}
\label{example_2}Similarly, $\mathsf{LCA}_{\mathbb{R}C}$ is not left (or
right) special in $\mathsf{LCA}$. If it were left special, there would be a
diagram%
\[%
%%
%& \mathbb{R}^n \oplus C \ar@{->>}[r] \ar[d] & \mathbb{Z}/2 \ar@{=}[d] \\
%\mathbb{Z} \ar@{^{(}->}[r] & \mathbb{Z} \ar@{->>}[r] & \mathbb{Z}/2,
%}}}%
%%
\xymatrix{
& \mathbb{R}^n \oplus C \ar@{->>}[r] \ar[d] & \mathbb{Z}/2 \ar@{=}[d] \\
\mathbb{Z} \ar@{^{(}->}[r] & \mathbb{Z} \ar@{->>}[r] & \mathbb{Z}/2,
}%
\]
but the image of a connected or a compact group in $\mathbb{Z}$ must be zero.
For right special, we can use the right side of Diagram \ref{DDE1} again.
\end{example}

Dually, these examples prove that $\mathsf{LCA}_{D}$ and $\mathsf{LCA}%
_{\mathbb{R}D}$ are neither left nor right special in $\mathsf{LCA}$.\medskip

Let us move on to examples which \textit{do} satisfy Schlichting's conditions.
In fact, it turns out that it suffices to enlarge the categories of the
previous examples slightly. Let $F$ be a number field and $\mathcal{O} $ its
ring of integers.

\begin{lemma}
The full subcategory $\mathsf{LCA}_{\mathcal{O},cg}$ is left filtering in
$\mathsf{LCA}_{\mathcal{O}}$.
\end{lemma}

\begin{proof}
Let $f:G\rightarrow X$ be an arbitrary morphism in $\mathsf{LCA}_{\mathcal{O}%
}$ with $G$ compactly generated. It is clear that $f$ can be factored as%
\[
G\overset{f}{\longrightarrow}\overline{\operatorname*{im}%
\nolimits_{\mathsf{Set}}(f)}\hookrightarrow X\text{,}%
\]
where $\overline{\operatorname*{im}\nolimits_{\mathsf{Set}}(f)}$ is the
closure of the set-theoretic image of $f$ in the topology of $X$. By this
choice, $\overline{\operatorname*{im}\nolimits_{\mathsf{Set}}(f)}%
\hookrightarrow X$ is an admissible monic. This proves our claim once we will
have shown that $\overline{\operatorname*{im}\nolimits_{\mathsf{Set}}(f)}$ is
compactly generated: As $G$ is compactly generated, there exists some
symmetric compact subset $U\subseteq G$ such that $G=\bigcup_{n\geq1}U^{n}$.
Thus, $f(U)$ is compact and $\operatorname*{im}\nolimits_{\mathsf{Set}%
}(f)=\bigcup_{n\geq1}f(U)^{n}$, so $f(U)$ is dense in $\overline
{\operatorname*{im}\nolimits_{\mathsf{Set}}(f)}$. Secondly, $\overline
{\operatorname*{im}\nolimits_{\mathsf{Set}}(f)}$, as a closed subgroup of a
locally compact group, is itself locally compact. Hence, we can find an open
neighbourhood $V$ of zero such that $\overline{V}$ is compact. Thus,
$\overline{V}\cup f(U)$ is compact and (in multiplicative notation)%
\[
\overline{\operatorname*{im}\nolimits_{\mathsf{Set}}(f)}=V \cdot \left(
\bigcup_{n\geq1}f(U)^{n}\right)  \subseteq\left(  \bigcup_{n\geq1}\left(
\overline{V}\cup f(U)\right)  ^{n}\right)  \subseteq\overline
{\operatorname*{im}\nolimits_{\mathsf{Set}}(f)}%
\]
The equality on the left holds since, on its right side, one summand is open
and the other dense. Hence, we have equality, and it follows that
$\overline{V}\cup f(U)$ generates the closure.
\end{proof}

I have learnt this argument from an article of Fujita and Shakhmatov, which
gives a detailed study of topological groups with a dense compactly generated
subgroup \cite[Theorem 2.7]{MR1931255}.

\begin{lemma}
The full subcategory $\mathsf{LCA}_{\mathcal{O},nss}$ is right filtering in
$\mathsf{LCA}_{\mathcal{O}}$.
\end{lemma}

\begin{proof}
This is essentially dual to the previous argument; we give the details
nonetheless: Let $f:X\rightarrow G$ be an arbitrary morphism in $\mathsf{LCA}%
_{\mathcal{O}}$ such that $G$ has no small subgroups. Then $f^{\vee}:G^{\vee
}\rightarrow X^{\vee}$ factors as%
\[
G^{\vee}\twoheadrightarrow G^{\vee}/\ker(f^{\vee})\longrightarrow
\overline{\operatorname*{im}\nolimits_{\mathsf{Set}}(f^{\vee})}\hookrightarrow
X^{\vee}%
\]
and the middle arrow has dense set-theoretic image. As $G$ has no small
subgroups, $G^{\vee}$ is compactly generated, cf. Proposition
\ref{prop_MoskowitzInterchange}, then so is its admissible quotient $G^{\vee
}/\ker(f^{\vee})$ and thus $\overline{\operatorname*{im}%
\nolimits_{\mathsf{Set}}(f^{\vee})}$ has a dense subgroup which is compactly
generated. By the same argument as in the previous proof, $\overline
{\operatorname*{im}\nolimits_{\mathsf{Set}}(f^{\vee})}$ is compactly
generated, but%
\[
(X/\ker(f))^{\vee}\cong\ker(f)^{\perp}\qquad\text{and}\qquad\ker(f)^{\perp
}\cong\overline{\operatorname*{im}\nolimits_{\mathsf{Set}}(f^{\vee})}%
\]
(this is standard, but see \cite[Chapter 2, \S 9.1 and \S 9.2]{MR1622489} for
a review of such identities). Thus, $(X/\ker(f))^{\vee}$ is compactly
generated, and therefore we obtain a factorization $X\twoheadrightarrow
X/\ker(f)\rightarrow G$, showing that $\mathsf{LCA}_{\mathcal{O},nss}$ is
right filtering.
\end{proof}

\begin{lemma}
\label{lemma_LCACGLeftSpecialInLCA}The full subcategory $\mathsf{LCA}%
_{\mathcal{O},cg}$ is left special in $\mathsf{LCA}_{\mathcal{O}}$.
\end{lemma}

\begin{proof}
(Step 1)\ We first show a more restrictive statement. Suppose%
\[
G^{\prime\prime}\hookrightarrow G\twoheadrightarrow C
\]
is an exact sequence with $C$ compact. We apply the Structure Theorem to $G$,
in the form of Remark \ref{rmk_SplitOffR}, and get a diagram%
\begin{equation}%
%%
%H \ar@{^{(}->}[d]_{i} \ar@{-->}[dr]^{h} \\
%{\bigoplus_{\sigma\in I}\mathbb{R}_{\sigma}} \oplus\tilde{G} \ar@{->>}%
%[r] \ar@{->>}[d] & C \ar@{-->>}[dr] \\
%{\bigoplus_{\sigma\in I}\mathbb{R}_{\sigma}} \oplus D & & C/{\operatorname
%{im}_{\mathsf{Set}}(h)}
%}} }%
%%
\xymatrix{
H \ar@{^{(}->}[d]_{i} \ar@{-->}[dr]^{h} \\
{\bigoplus_{\sigma\in I}\mathbb{R}_{\sigma}} \oplus\tilde{G} \ar@{->>}%
[r] \ar@{->>}[d] & C \ar@{-->>}[dr] \\
{\bigoplus_{\sigma\in I}\mathbb{R}_{\sigma}} \oplus D & & C/{\operatorname
{im}_{\mathsf{Set}}(h)}
}
\label{diag1}%
\end{equation}
with $H$ compact, $D$ discrete. As $H$ is compact, the image of $h$ is
compact, so the set-theoretic image $\operatorname*{im}\nolimits_{\mathsf{Set}%
}(h)$ agrees with the categorical image, and is a closed subgroup of $C$. This
justifies why the second dashed diagonal arrow in the above diagram is an
admissible epic. By the universal property of cokernels we get a new diagram%
\[%
%%
%H \ar@{^{(}->}[d]_{i} \ar[dr]^{0} \\
%{\bigoplus_{\sigma\in I}\mathbb{R}_{\sigma}} \oplus\tilde{G} \ar@
%{->>}[r]^{\gamma} \ar@{->>}[d] & C/{\operatorname{im}(h)} \\
%{\bigoplus_{\sigma\in I}\mathbb{R}_{\sigma}} \oplus D \ar@{-->}[ur]_{w}
%}}}%
%%
\xymatrix{
H \ar@{^{(}->}[d]_{i} \ar[dr]^{0} \\
{\bigoplus_{\sigma\in I}\mathbb{R}_{\sigma}} \oplus\tilde{G} \ar@
{->>}[r]^{\gamma} \ar@{->>}[d] & C/{\operatorname{im}(h)} \\
{\bigoplus_{\sigma\in I}\mathbb{R}_{\sigma}} \oplus D \ar@{-->}[ur]_{w}
}%
\]
By \cite[Cor. 7.7]{MR2606234} the morphism $w$ must be an admissible epic. Let
us, temporarily, just consider $w$ from the viewpoint of the underlying
locally compact additive groups: Regarded in $\mathsf{LCA}$, $w$ remains an
admissible epic, and is of the shape%
\[
w:\mathbb{R}^{n}\oplus D\twoheadrightarrow C/\operatorname*{im}(h)\text{.}%
\]
All admissible quotients of $\mathbb{R}^{n}\oplus D$ have the shape
$\mathbb{R}^{i}\oplus\mathbb{T}^{j}\oplus D^{\prime}$ with $D^{\prime}$
discrete (this follows from the classification of the closed subgroups of
$\mathbb{R}^{n}\oplus D$, see Corollary 2 to Theorem 7 in \cite{MR0442141}).
On the other hand, $C/\operatorname*{im}(h)$ is compact, so its underlying
additive topological group can only have the shape $\mathbb{T}^{j}\oplus
D^{\prime}$ for $D^{\prime}$ finite. We return to $\mathsf{LCA}_{\mathcal{O}}%
$: By the previous argument $C/\operatorname*{im}(h)$ has no small subgroups,
so by Theorem \ref{thm_ResolvableObjects} it has a projective resolution. In
fact, as the proof of the theorem loc. cit. explains, we find an admissible
epic of the shape%
\begin{equation}
\beta_{0}:\bigoplus_{\nu\in{I^{\prime}}}\mathbb{R}_{\nu}\oplus\bigoplus
_{J\in\mathcal{I}}J\twoheadrightarrow C/\operatorname*{im}(h) \label{lwa1}%
\end{equation}
for ${I^{\prime}}$ a finite list of real and complex places, while
$\mathcal{I}$ is a finite list of ideals of $\mathcal{O}$, and the object on
the left is projective in $\mathsf{LCA}_{\mathcal{O}}$. The source being
projective, this morphism $\beta_{0}$ can be lifted along the epic $\gamma$ to
some morphism $\beta:\bigoplus_{\nu\in{I^{\prime}}}\mathbb{R}_{\nu}%
\oplus\bigoplus_{J\in\mathcal{I}}J\rightarrow\bigoplus_{\sigma\in I}%
\mathbb{R}_{\sigma}\oplus\tilde{G}$. Returning to Diagram \ref{diag1}, we now
get%
\[%
%%
%H \oplus{\bigoplus_{\nu\in{I^{\prime}}}\mathbb{R}_{\nu}\oplus\bigoplus
%_{J\in\mathcal{I}}J} \ar[d]_{i+{\beta}} \ar@{-->}[dr]^{\tilde{h}} \\
%{\bigoplus_{\sigma\in I}\mathbb{R}_{\sigma}} \oplus\tilde{G} \ar@{->>}[r] & C,
%}}}%
%%
\xymatrix{
H \oplus{\bigoplus_{\nu\in{I^{\prime}}}\mathbb{R}_{\nu}\oplus\bigoplus
_{J\in\mathcal{I}}J} \ar[d]_{i+{\beta}} \ar@{-->}[dr]^{\tilde{h}} \\
{\bigoplus_{\sigma\in I}\mathbb{R}_{\sigma}} \oplus\tilde{G} \ar@{->>}[r] & C,
}%
\]
where $\tilde{h}$ is defined on $H$ by $\left.  \tilde{h}\mid_{H}\right.  :=h$
(this refers to $h$ in Diagram \ref{diag1}), and on $\bigoplus_{\nu
\in{I^{\prime}}}\mathbb{R}_{\nu}\oplus\bigoplus_{J\in\mathcal{I}}J$ the map
$\tilde{h}$ is defined by following the other arrows of the triangle. Finally,
note that $\tilde{h}$ is, by construction, continuous and surjective, and
since $I^{\prime},\mathcal{I}$ are finite, its source is $\sigma$-compact, so
by Pontryagin's Open Mapping\ Theorem the map $\tilde{h} $ is also open. It
follows that $\tilde{h}$ is an admissible epic. Moreover, $H\oplus
\bigoplus_{\nu\in I^{\prime}}\mathbb{R}_{\nu}\oplus\bigoplus_{J\in\mathcal{I}%
}J$ is compactly generated.

As a result, we obtain the commutative diagram%
\begin{equation}%
%%
%K \ar@{^{(}->}[r] \ar@{-->}[d] & H \oplus{\bigoplus_{\nu\in I^{\prime}}%
%\mathbb{R}_{\nu}\oplus\bigoplus_{J\in\mathcal{I}}J} \ar[d]_{i+{\beta}}
%\ar@{->>}[r]^-{\tilde{h}} & C \ar@{=}[d] \\
%G^{\prime\prime} \ar@{^{(}->}[r] & G \ar@{->>}[r] & C.
%}} }%
%%
\xymatrix{
K \ar@{^{(}->}[r] \ar@{-->}[d] & H \oplus{\bigoplus_{\nu\in I^{\prime}}%
\mathbb{R}_{\nu}\oplus\bigoplus_{J\in\mathcal{I}}J} \ar[d]_{i+{\beta}}
\ar@{->>}[r]^-{\tilde{h}} & C \ar@{=}[d] \\
G^{\prime\prime} \ar@{^{(}->}[r] & G \ar@{->>}[r] & C.
}
\label{diag2}%
\end{equation}
Here $K:=\ker(\tilde{h})$ and the dashed arrow exists by universal property.
Finally, $K$ is also compactly generated, because it is a closed subgroup of a
compactly generated group \cite[Theorem 2.6]{MR0215016}. It follows that the
top row lies in $\mathsf{LCA}_{\mathcal{O},cg}$.\newline(Step 2)\ Now let%
\begin{equation}
G^{\prime\prime}\hookrightarrow\hat{G}\twoheadrightarrow G^{\prime}
\label{lw1}%
\end{equation}
be an arbitrary exact sequence in $\mathsf{LCA}_{\mathcal{O}}$ with
$G^{\prime}$ compactly generated. By Proposition
\ref{prop_MoskClassifNSSandCGGroups} we have%
\[
G^{\prime}\simeq P\oplus C\qquad\text{with}\qquad P:=\bigoplus_{\sigma\in
I^{\prime}}\mathbb{R}_{\sigma}\oplus\bigoplus_{J\in\mathcal{I}^{\prime}}J
\]
with $C$ a compact $\mathcal{O}$-module, $\mathcal{I}^{\prime}$ a finite list
of ideals in $\mathcal{O}$, and $I^{\prime}$ a finite list of real and complex
places. The object $P$ is both compactly generated and projective by Theorem
\ref{thm_ProjectivesInLCAO}. Thus, it can be split off as a direct summand.
Hence, we get an exact sequence%
\begin{equation}
G^{\prime\prime}\hookrightarrow G\twoheadrightarrow C \label{lw2}%
\end{equation}
for some complement $G$ of $P$ in $\hat{G}$, and the exact sequence in
Equation \ref{lw1} is isomorphic to $G^{\prime\prime}\hookrightarrow P\oplus
G\twoheadrightarrow P\oplus C$. Apply Step 1 to Equation \ref{lw2}, i.e. we
arrive at Diagram \ref{diag2}. Now, using the projectivity of $P$ and the
resulting splittings, add the summand $P$ compatibly, giving%
\[%
%%
%K \ar@{^{(}->}[r] \ar@{->}[d] & P
%\oplus H \oplus{\bigoplus_{\nu\in J}\mathbb{R}_{\nu}\oplus\bigoplus
%_{I\in\mathcal{I}}I} \ar[d] \ar@{->>}[r] & P \oplus C \ar@{=}[d] \\
%G^{\prime\prime} \ar@{^{(}->}[r] & P
%\oplus G \ar@{->>}[r] & P \oplus C
%}}}%
%%
\xymatrix{
K \ar@{^{(}->}[r] \ar@{->}[d] & P
\oplus H \oplus{\bigoplus_{\nu\in J}\mathbb{R}_{\nu}\oplus\bigoplus
_{I\in\mathcal{I}}I} \ar[d] \ar@{->>}[r] & P \oplus C \ar@{=}[d] \\
G^{\prime\prime} \ar@{^{(}->}[r] & P
\oplus G \ar@{->>}[r] & P \oplus C
}%
\]
Again, all groups in the top row are compactly generated. We conclude that
$\mathsf{LCA}_{\mathcal{O},cg}$ is left special in $\mathsf{LCA}_{\mathcal{O}}
$.
\end{proof}

\begin{lemma}
\label{lemma_LCAnssRightSpecial}The full subcategory $\mathsf{LCA}%
_{\mathcal{O},nss}$ is right special in $\mathsf{LCA}_{\mathcal{O}}$.
\end{lemma}

Using Proposition \ref{prop_MoskowitzInterchange}, this is the Pontryagin dual
statement to Lemma \ref{lemma_LCACGLeftSpecialInLCA}.\medskip

We can obtain some further variations with no extra work:

\begin{lemma}
\label{lemma_LCAC_leftsfiltInLCARC}The full subcategory $\mathsf{LCA}%
_{\mathcal{O},C}$ is left $s$-filtering in $\mathsf{LCA}_{\mathcal{O}%
,\mathbb{R}C}$.
\end{lemma}

\begin{proof}
We first prove left special: We use exactly the same proof as of Lemma
\ref{lemma_LCACGLeftSpecialInLCA} (we only need Step 1 this time), and since
we now are in $\mathsf{LCA}_{\mathcal{O},\mathbb{R}C}$, we may additionally
assume $D=0$. So we must have $D^{\prime}=0$, and thus can run the argument
with $\mathcal{I}=\varnothing$ in Equation \ref{lwa1}, giving the claim. Left
filtering is clear:\ Given any $f:G\rightarrow X$ with $G$ compact, it factors
as $G\rightarrow\operatorname*{im}\nolimits_{\mathsf{Set}}(f)\rightarrow X $,
but since $G$ is compact, its image $\operatorname*{im}\nolimits_{\mathsf{Set}%
}(f)$ is also compact, thus closed, i.e. $\operatorname*{im}%
\nolimits_{\mathsf{Set}}(f)\hookrightarrow X$ is an admissible monic.
\end{proof}

\begin{lemma}
\label{lemma_LCAD_rightsfiltInLCARD}The full subcategory $\mathsf{LCA}%
_{\mathcal{O},D}$ is right $s$-filtering in $\mathsf{LCA}_{\mathcal{O}%
,\mathbb{R}D}$.
\end{lemma}

Again, this is just the Pontryagin dual statement.

\begin{lemma}
\label{lemma_FLCARDextclosed}The category $\mathsf{LCA}_{\mathcal{O}%
,\mathbb{R}D}$ is extension-closed in $\mathsf{LCA}_{\mathcal{O}}$.
\end{lemma}

We begin by proving a special case of the claim in $\mathsf{LCA}$, i.e. for
$\mathcal{O}=\mathbb{Z}$:

\begin{lemma}
\label{lemma_LCARD_closedunderextbydiscretes}Suppose $G^{\prime}%
\hookrightarrow G\twoheadrightarrow N$ is an exact sequence in $\mathsf{LCA} $
with $G^{\prime}\in\mathsf{LCA}_{\mathbb{R}D}$ and $N$ discrete, then
$G\in\mathsf{LCA}_{\mathbb{R}D}$.
\end{lemma}

\begin{proof}
Consider the extension $\mathbb{R}^{i}\oplus D^{\prime}\hookrightarrow
G\twoheadrightarrow N$ with $D^{\prime}$ discrete. We get the filtration
$\mathbb{R}^{i}\hookrightarrow\mathbb{R}^{i}\oplus D^{\prime}\hookrightarrow
G$. The variant of Noether's Lemma for exact categories \cite[Lemma
3.5]{MR2606234}, yields the exact sequence $(\mathbb{R}^{i}\oplus D^{\prime
})/\mathbb{R}^{i}\hookrightarrow G/\mathbb{R}^{i}\twoheadrightarrow
G/(\mathbb{R}^{i}\oplus D^{\prime})$ and this simplifies to $D^{\prime
}\hookrightarrow G/\mathbb{R}^{i}\twoheadrightarrow N$. As $D^{\prime}$ and
$N$ are discrete and discrete groups are closed under extension, it follows
that we have an exact sequence $\mathbb{R}^{i}\hookrightarrow
G\twoheadrightarrow D$ with $D$ discrete. However, $\mathbb{R}^{i}$ is an
injective object and thus this sequence must split. Hence, $G\simeq
\mathbb{R}^{i}\oplus D$ with $D$ discrete, giving the claim.
\end{proof}

\begin{proof}
[Proof of Lemma \ref{lemma_FLCARDextclosed}](Step 1) First, we prove this for
$\mathcal{O}=\mathbb{Z}$, i.e. in $\mathsf{LCA}$. This is tedious to spell
out, but quite easy: Let $\mathbb{R}^{i}\oplus D^{\prime}\hookrightarrow
G\twoheadrightarrow\mathbb{R}^{j}\oplus D^{\prime\prime}$ be an extension with
$D^{\prime},D^{\prime\prime}$ discrete groups. By the Structure Theorem
(Theorem \ref{thm_StructOfLCA}), $G$ decomposes as $\mathbb{R}^{n}\oplus
H\hookrightarrow G\twoheadrightarrow D$, $H$ compact, $D$ discrete. Consider
the composition $h:H\hookrightarrow G\twoheadrightarrow\mathbb{R}^{j}\oplus
D^{\prime\prime}$. As $H$ is compact, the image of $H$ after projecting to
$\mathbb{R}^{j}$ is compact and thus trivial. Hence, the map factors as
$H\rightarrow D^{\prime\prime}$, and since $H$ is compact but $D^{\prime
\prime}$ discrete, the image $\operatorname*{im}\nolimits_{\mathsf{Set}}(h)$
must be finite. Thus, we get a factorization%
\begin{equation}
\xymatrix{ & N \ar@{^{(}->}[d] \\ H \ar@{^{(}->}[r] \ar[ur] \ar[dr]_{0} & G \ar@{->>}[d] \\ & (\mathbb{R}^{j} \oplus D^{\prime \prime}) / {\operatorname{im}(h)}\text{,} }
\label{dd1}%
\end{equation}
where $N$ is defined as the kernel of the downward admissible epic originating
from $G$. As this quotients out more from $G$, we get a filtration
$\mathbb{R}^{i}\oplus D^{\prime}\hookrightarrow N\hookrightarrow G$ and by
Noether's Lemma, the corresponding sequence of quotients $\frac{N}%
{\mathbb{R}^{i}\oplus D^{\prime}}\hookrightarrow\frac{G}{\mathbb{R}^{i}\oplus
D^{\prime}}\twoheadrightarrow\frac{G}{N}$ is exact. As $G/N\simeq
(\mathbb{R}^{j}\oplus D^{\prime\prime})/\operatorname*{im}%
\nolimits_{\mathsf{Set}}(h)$ by the choice of $N$, and $G/(\mathbb{R}%
^{i}\oplus D^{\prime})\simeq\mathbb{R}^{j}\oplus D^{\prime\prime}$, we learn
that $N/(\mathbb{R}^{i}\oplus D^{\prime})\simeq\operatorname*{im}%
\nolimits_{\mathsf{Set}}(h)$. Hence, $\mathbb{R}^{i}\oplus D^{\prime
}\hookrightarrow N\twoheadrightarrow\operatorname*{im}\nolimits_{\mathsf{Set}%
}(h)$ is exact.\ It follows that $N$ is an extension of a group in
$\mathsf{LCA}_{\mathbb{R}D}$ by a finite group, i.e. $N\simeq\mathbb{R}%
^{i}\oplus E$ with $E$ discrete by Lemma
\ref{lemma_LCARD_closedunderextbydiscretes}. In Diagram \ref{dd1} the
factorization $H\rightarrow N$ is an admissible monic by \cite[Corollary
7.7]{MR2606234} and thus it becomes $H\hookrightarrow\mathbb{R}^{i}\oplus E$.
As $H$ is compact and $\mathbb{R}^{i}$ has no non-trivial compact subgroups,
$H$ must lie in $E$, but the latter is discrete, so $H$ must be a finite
group. Thus, we have $\mathbb{R}^{n}\oplus H\hookrightarrow
G\twoheadrightarrow D$ with $D$ discrete and $H$ finite. Now Lemma
\ref{lemma_LCARD_closedunderextbydiscretes} implies that $G$ lies in
$\mathsf{LCA}_{\mathbb{R}D}$. (Step 2) Finally, the case of general
$\mathcal{O}$ reduces to the case over $\mathbb{Z}$ by Lemma
\ref{lemma_VectorGroupDiscreteInLCALiftsToLCAO}.
\end{proof}

\begin{example}
The morphism $\mathbb{Z}\rightarrow\mathbb{R}$ is an element of
$\operatorname*{Hom}(\mathbb{Z},\mathbb{R})$, both in $\mathsf{LCA}$ as well
as in $\mathsf{LCA}_{\mathbb{R}D}$. It is an admissible monic in
$\mathsf{LCA}$, but not in $\mathsf{LCA}_{\mathbb{R}D}$ as the quotient
$\mathbb{T}$ does not lie in $\mathsf{LCA}_{\mathbb{R}D}$.
\end{example}

\begin{lemma}
\label{lemma_FLCclosed}The category $\mathsf{LCA}_{\mathcal{O},\mathbb{R}C}$
is extension-closed in $\mathsf{LCA}_{\mathcal{O}}$.
\end{lemma}

\begin{corollary}
All the full subcategories $\mathsf{LCA}_{\mathcal{O},C}$, $\mathsf{LCA}%
_{\mathcal{O},D}$, $\mathsf{LCA}_{\mathcal{O},\mathbb{R}D}$, $\mathsf{LCA}%
_{\mathcal{O},\mathbb{R}C}$, $\mathsf{LCA}_{\mathcal{O},cg}$, $\mathsf{LCA}%
_{\mathcal{O},nss}$ are fully exact subcategories of $\mathsf{LCA}%
_{\mathcal{O}}$. In particular, they all carry a canonical exact structure and
are closed under extensions in $\mathsf{LCA}_{\mathcal{O}}$. A sequence is
exact in these categories if and only if it is exact in $\mathsf{LCA}$; that
is: The forgetful functor%
\[
\mathsf{LCA}_{\mathcal{O},(-)}\longrightarrow\mathsf{LCA}%
\]
is faithful, exact and reflects exactness.
\end{corollary}

\begin{proof}
For this type of claim, thanks to \cite[Lemma 10.20]{MR2606234} we only need
to know that the full subcategory in question is closed under extensions. For
$\mathsf{LCA}_{cg}$, $\mathsf{LCA}_{nss}$ this was shown by Moskowitz, see
Proposition \ref{prop_MoskowitzInterchange} (we could take an alternative
route here: By \cite[Lemma 2.14]{MR3510209} left special subcategories are
closed under extensions. Dually, the same is true for right special
subcategories. This would also allow us to handle $\mathsf{LCA}_{cg}$,
$\mathsf{LCA}_{nss}$ by Lemma \ref{lemma_LCACGLeftSpecialInLCA} and Lemma
\ref{lemma_LCAnssRightSpecial}). The claims for $\mathsf{LCA}_{\mathcal{O}%
,cg}$, $\mathsf{LCA}_{\mathcal{O},nss}$ follow from this as they refer to
properties which only depend on the underlying additive group. For
$\mathsf{LCA}_{\mathcal{O},D}$ in $\mathsf{LCA}$ note that any extension of
discrete $\mathcal{O}$-modules is discrete, so it is extension-closed in
$\mathsf{LCA}_{\mathcal{O}}$. Taking the Pontryagin dual statement, we get
that $\mathsf{LCA}_{\mathcal{O},C}$ is closed under extensions. For
$\mathsf{LCA}_{\mathcal{O},\mathbb{R}D}$, $\mathsf{LCA}_{\mathcal{O}%
,\mathbb{R}C}$ we had already shown that they are closed under extensions
(Lemma \ref{lemma_FLCARDextclosed}, Lemma \ref{lemma_FLCclosed}). The functor
$\mathsf{LCA}_{\mathcal{O},(-)}\longrightarrow\mathsf{LCA}$ is exact and
reflects exactness as a corollary of the exact structures on all these
categories being defined through \cite[Lemma 10.20]{MR2606234}. The
faithfulness is clear.
\end{proof}

\section{Proof of\ the main theorem}

Let us quickly recall how exact categories and the notions of left (resp.
right) $s$-filtering are connected to stable $\infty$-categories. Firstly,
Schlichting has shown that the natural notion of an exact sequence of exact
categories induces an exact sequence of homotopy categories.

\begin{theorem}
[Schlichting Localization]\label{thm_SchlichtingLocalization}Let $\mathsf{C}$
be an idempotent complete exact category and $\mathsf{C}\hookrightarrow
\mathsf{D}$ a right (or left) $s$-filtering inclusion into an exact category
$\mathsf{D}$. Then%
\[
D^{b}(\mathsf{C})\hookrightarrow D^{b}(\mathsf{D})\twoheadrightarrow
D^{b}(\mathsf{D}/\mathsf{C})
\]
is an exact sequence of triangulated categories.
\end{theorem}

This is \cite[Prop. 2.6]{MR2079996}. See Keller \cite[\S 11]{MR1421815} or
B\"{u}hler \cite[\S 10]{MR2606234} for the construction of the category of
chain complexes and the homotopy category. Secondly, exactness on the homotopy
category level is essentially equivalent to all other conceivable notions of
exactness, see the discussion in \cite[\S 5]{MR3070515}.

Next, recall the Eilenberg swindle. Although absolutely classical, we record
it because it is so crucial for the computation:

\begin{lemma}
\label{lemma_EilenbergSwindle}If an exact category $\mathsf{C}$ is closed
under countable products or coproducts and $K:\operatorname*{Cat}_{\infty
}^{\operatorname*{ex}}\rightarrow\mathsf{A}$ is any additive invariant (e.g.,
a localizing invariant), then $K(\mathsf{C})=0$.
\end{lemma}

\begin{proof}
Suppose $\mathsf{C}$ has countable coproducts. The Eilenberg swindle functor
is defined by%
\[
E:\mathsf{C}\longrightarrow\mathcal{E}\mathsf{C}\text{,}\qquad X\mapsto
(X\hookrightarrow%
{\textstyle\coprod\nolimits_{i\in\mathbb{Z}}}
X\twoheadrightarrow%
{\textstyle\coprod\nolimits_{i\in\mathbb{Z}}}
X)\text{,}%
\]
where $\mathcal{E}\mathsf{Ab}$ denotes the exact category of exact sequences
and the epic is the `shift one to the left' functor. Projecting to the left,
middle and right term of $\mathcal{E}\mathsf{C}$, call these $p_{i}%
:\mathsf{C}\rightarrow\mathsf{C}$ with $i=0,1,2$, the additivity of $K$
implies that $p_{1}=p_{0}+p_{2}$ holds for the maps induced on $K(\mathsf{C}%
)\rightarrow K(\mathsf{C})$, but $p_{1}=p_{2}$, so clearly $p_{0}=0$. However,
$p_{0}$ is just the identity functor. If $\mathsf{C}$ has countable products,
perform the same trick with products instead.
\end{proof}

\begin{lemma}
\label{lemma_KOfLCAnss_or_cg}For every localizing invariant
$K:\operatorname*{Cat}_{\infty}^{\operatorname*{ex}}\rightarrow\mathsf{A}$,
there are canonical equivalences%
\[
K(\mathsf{LCA}_{\mathcal{O},nss})\overset{\sim}{\longrightarrow}%
K(\mathsf{LCA}_{\mathcal{O},\mathbb{R}})\qquad\text{and}\qquad K(\mathsf{LCA}%
_{\mathcal{O},cg})\overset{\sim}{\longrightarrow}K(\mathsf{LCA}_{\mathcal{O}%
,\mathbb{R}})\text{.}%
\]
Taking Pontryagin duals on the left and vector space duals on the right
transforms one equivalence into the other.
\end{lemma}

\begin{proof}
(Step 1) Firstly, we establish equivalences%
\[
K(\mathsf{LCA}_{\mathcal{O},\mathbb{R}D})\overset{\sim}{\longrightarrow
}K(\mathsf{LCA}_{\mathcal{O},nss})\qquad\text{and}\qquad K(\mathsf{LCA}%
_{\mathcal{O},\mathbb{R}C})\overset{\sim}{\longrightarrow}K(\mathsf{LCA}%
_{\mathcal{O},cg})\text{.}%
\]
Consider $\mathsf{LCA}_{\mathcal{O},nss}$: These modules admit a projective
resolution, see Theorem \ref{thm_ResolvableObjects}. The resolution then has
the shape $P_{1}\hookrightarrow P_{0}\twoheadrightarrow G$. By the
classification of projectives in $\mathsf{LCA}_{\mathcal{O}}$, Theorem
\ref{thm_ProjectivesInLCAO}, for both $i=0,1$ we have an isomorphism%
\[
P_{i}\simeq\bigoplus_{\sigma\in I_{i}}\mathbb{R}_{\sigma}\oplus\prod
_{J\in\mathcal{I}_{i}}J\text{,}%
\]
where $I_{i}$ is a finite list of real and complex places and $\mathcal{I}%
_{i}$ a (possibly infinite) list of ideals in $\mathcal{O}$. In particular,
every object in $\mathsf{LCA}_{\mathcal{O},nss}$ has a two-term resolution by
projectives in $\mathsf{LCA}_{\mathcal{O},\mathbb{R}D}$. The claim for
$\mathsf{LCA}_{\mathcal{O},cg}$ is the Pontryagin dual statement.\newline(Step
2) Lemma \ref{lemma_LCAD_rightsfiltInLCARD} shows that $\mathsf{LCA}%
_{\mathcal{O},D}\hookrightarrow$ $\mathsf{LCA}_{\mathcal{O},\mathbb{R}D}$ is
right $s$-filtering. Thus, the quotient exact category exists and it is easy
to see that the natural functor%
\[
\mathsf{LCA}_{\mathcal{O},\mathbb{R}}\overset{\sim}{\longrightarrow
}\mathsf{LCA}_{\mathcal{O},\mathbb{R}D}/\mathsf{LCA}_{\mathcal{O},D}%
\]
is an exact equivalence of exact categories\ (it is fully faithful, clearly
exact and essentially surjective). Now, by Schlichting's Localization Theorem,
$K(\mathsf{LCA}_{\mathcal{O},D})\rightarrow K(\mathsf{LCA}_{\mathcal{O}%
,\mathbb{R}D})\rightarrow K(\mathsf{LCA}_{\mathcal{O},\mathbb{R}}) $ is a
fiber sequence, but by the Eilenberg swindle, the left term is zero, Lemma
\ref{lemma_EilenbergSwindle}. Finally, dualization of the entire argument
handles $\mathsf{LCA}_{\mathcal{O},cg}$.
\end{proof}

As $\mathsf{LCA}_{\mathcal{O},cg}\hookrightarrow\mathsf{LCA}_{\mathcal{O}}$ is
left filtering and left special, we get an exact sequence of exact categories%
\[
\mathsf{LCA}_{\mathcal{O},cg}\hookrightarrow\mathsf{LCA}_{\mathcal{O}%
}\twoheadrightarrow\mathsf{LCA}_{\mathcal{O}}/\mathsf{LCA}_{\mathcal{O}%
,cg}\text{.}%
\]
In particular, the quotient $\mathsf{LCA}_{\mathcal{O}}/\mathsf{LCA}%
_{\mathcal{O},cg}$ exists in the $2$-category of exact categories. Next, let
$\mathsf{Mod}(\mathcal{O})$ be the category of all (algebraic) $\mathcal{O}%
$-modules, and $\mathsf{Mod}_{fg}(\mathcal{O})$ the subcategory of finitely
generated $\mathcal{O}$-modules. Note that $\mathsf{Mod}{_{fg}(\mathcal{O}%
)}\hookrightarrow\mathsf{Mod}(\mathcal{O})$ is also left filtering and left
special\footnote{It is a Serre subcategory of an abelian category, so it is
automatically left and right filtering. Left specialness is easy to check.}.
We obtain an exact sequence of abelian categories, and indeed get a
commutative diagram%
\begin{equation}%
%%
%{\mathsf{Mod}_{fg}(\mathcal{O})} \ar@{^{(}->}[r] \ar[d] & {\mathsf
%{Mod}(\mathcal{O})} \ar@{->>}[r] \ar[d] & {\mathsf{Mod}(\mathcal{O})}%
%/{\mathsf{Mod}_{fg}(\mathcal{O})} \ar[d] \\
%\mathsf{LCA}_{\mathcal{O},cg} \ar@{^{(}->}[r] & \mathsf{LCA}_{\mathcal{O}}
%\ar@{->>}[r] & {\mathsf{LCA}_{\mathcal{O}}}/{\mathsf{LCA}_{\mathcal{O},cg}}.
%}} }%
%%
\xymatrix{
{\mathsf{Mod}_{fg}(\mathcal{O})} \ar@{^{(}->}[r] \ar[d] & {\mathsf
{Mod}(\mathcal{O})} \ar@{->>}[r] \ar[d] & {\mathsf{Mod}(\mathcal{O})}%
/{\mathsf{Mod}_{fg}(\mathcal{O})} \ar[d] \\
\mathsf{LCA}_{\mathcal{O},cg} \ar@{^{(}->}[r] & \mathsf{LCA}_{\mathcal{O}}
\ar@{->>}[r] & {\mathsf{LCA}_{\mathcal{O}}}/{\mathsf{LCA}_{\mathcal{O},cg}}.
}
\label{ld5}%
\end{equation}
The downward arrow just sends an algebraic $\mathcal{O}$-module to its
topological counterpart, equipped with the discrete topology. All arrows are
exact functors.

\begin{lemma}
\label{lemma_equivq}The functor $\mathsf{Mod}{(\mathcal{O})}/\mathsf{Mod}%
{_{fg}(\mathcal{O})}\longrightarrow\mathsf{LCA}_{\mathcal{O}}/\mathsf{LCA}%
_{\mathcal{O},cg}$ is an exact equivalence of exact categories. Both
categories are abelian.
\end{lemma}

\begin{proof}
The functor is essentially surjective: Given any $G$ in $\mathsf{LCA}%
_{\mathcal{O}}$, the Structure Theorem (Theorem \ref{thm_StructOfLCA})
presents it as%
\[
\bigoplus_{\sigma\in I}\mathbb{R}_{\sigma}\oplus C\hookrightarrow
G\twoheadrightarrow D
\]
with $C$ compact, $D$ discrete and $I$ a finite list. As $\bigoplus_{\sigma\in
I}\mathbb{R}_{\sigma}\oplus C$ is compactly generated, $G\twoheadrightarrow D$
becomes an isomorphism in $\mathsf{LCA}_{\mathcal{O}}/\mathsf{LCA}%
_{\mathcal{O},cg}$. As all modules in the essential image carry the discrete
topology, there is no difference between arbitrary module homomorphisms or
continuous ones, so the functor is moreover fully faithful and exact. As
$\mathsf{Mod}{(\mathcal{O})}/\mathsf{Mod}{_{fg}(\mathcal{O})}$ stems from a
quotient of abelian categories, it is itself abelian (note that we quotient by
a Serre subcategory). By the equivalence of categories, the same holds for
$\mathsf{LCA}_{\mathcal{O}}/\mathsf{LCA}_{\mathcal{O},cg}$.
\end{proof}

We are ready to prove our main result.

\begin{theorem}
\label{thm_MainThm}Let $F$ be a number field and $\mathcal{O}$ its ring of
integers. For every localizing invariant $K:\operatorname*{Cat}_{\infty
}^{\operatorname*{ex}}\rightarrow\mathsf{A}$, there is a canonical fiber
sequence%
\begin{equation}
K(\mathcal{O})\longrightarrow K(\mathbb{R})^{r}\oplus K(\mathbb{C}%
)^{s}\longrightarrow K(\mathsf{LCA}_{\mathcal{O}})\text{,} \label{la1}%
\end{equation}
where $r$ is the number of real places and $s$ the number of complex places.
\end{theorem}

\begin{proof}
By Schlichting's Localization Theorem, Theorem
\ref{thm_SchlichtingLocalization}, both exact rows in Diagram \ref{ld5} induce
fiber sequences in any localizing theory,%
\begin{equation}%
%%
%K({\mathsf{Mod}_{fg}(\mathcal{O})}) \ar[r] \ar[d]_{g} & K({\mathsf
%{Mod}(\mathcal{O})}) \ar[r] \ar[d] & K({{\mathsf{Mod}(\mathcal{O})}}%
%/{{\mathsf{Mod}_{fg}(\mathcal{O})}}) \ar[d]_{\sim} \\
%K(\mathsf{LCA}_{\mathcal{O},cg}) \ar[r] & K(\mathsf{LCA}_{\mathcal{O}}%
%) \ar[r] & K({\mathsf{LCA}_{\mathcal{O}}}/{\mathsf{LCA}_{\mathcal{O},cg}}),
%}} }%
%%
\xymatrix{
K({\mathsf{Mod}_{fg}(\mathcal{O})}) \ar[r] \ar[d]_{g} & K({\mathsf
{Mod}(\mathcal{O})}) \ar[r] \ar[d] & K({{\mathsf{Mod}(\mathcal{O})}}%
/{{\mathsf{Mod}_{fg}(\mathcal{O})}}) \ar[d]_{\sim} \\
K(\mathsf{LCA}_{\mathcal{O},cg}) \ar[r] & K(\mathsf{LCA}_{\mathcal{O}}%
) \ar[r] & K({\mathsf{LCA}_{\mathcal{O}}}/{\mathsf{LCA}_{\mathcal{O},cg}}),
}
\label{lda2}%
\end{equation}
and the equivalence on the right stems from Lemma \ref{lemma_equivq}. Thus,
the left square is a bi-cartesian square in $\mathsf{A}$. As the category
$\mathsf{Mod}(\mathcal{O})$ has countable coproducts, the Eilenberg swindle
yields $K(\mathsf{Mod}(\mathcal{O}))=0$, see Lemma
\ref{lemma_EilenbergSwindle}. Moreover, all objects in $\mathsf{Mod}%
_{fg}(\mathcal{O})$ have a short projective resolution by the projective
object $\mathcal{O}$, so $K(\mathsf{Mod}_{fg}(\mathcal{O}))\simeq
K(\mathcal{O})$. By Lemma \ref{lemma_KOfLCAnss_or_cg}, $K(\mathsf{LCA}%
_{\mathcal{O},cg})\cong K(\mathsf{LCA}_{\mathcal{O},\mathbb{R}})$, and by
Corollary \ref{cor_LCAOR} we have $K(\mathsf{LCA}_{\mathcal{O},\mathbb{R}%
})\cong K(\mathbb{R})^{r}\oplus K(\mathbb{C})^{s}$. Thus,%
\[
K(\mathcal{O})\longrightarrow K(\mathbb{R})^{r}\oplus K(\mathbb{C}%
)^{s}\longrightarrow K(\mathsf{LCA}_{\mathcal{O}})
\]
is a fiber sequence in $\mathsf{A}$, proving the claim.
\end{proof}

Finally, we identify the map:

\begin{theorem}
We keep the assumptions of the previous theorem. The map $K(\mathcal{O}%
)\longrightarrow K(\mathbb{R})^{r}\oplus K(\mathbb{C})^{s}$ in our
construction of Equation \ref{la1} is induced from the Minkowski embedding
ring homomorphism $\mathcal{O}\rightarrow\mathbb{R}^{r}\oplus\mathbb{C}^{s} $.
\end{theorem}

\begin{proof}
Inspecting Diagram \ref{ld5}, the exact functor underlying the map
$K(\mathcal{O})\longrightarrow K(\mathbb{R})^{r}\oplus K(\mathbb{C})^{s}$
sends the object $\mathcal{O}$ of $\mathsf{Mod}_{fg}(\mathcal{O})$ to
$\mathcal{O}$ in $\mathsf{LCA}_{\mathcal{O},cg}$. This looks a priori
different from the functor we want, however, both induce the same map on the
level of $K$: Let $\mathsf{Proj}(\mathsf{Mod}_{fg}(\mathcal{O}))$ denote the
fully exact subcategory of $\mathsf{Mod}_{fg}(\mathcal{O})$ of projectives,
generated by the single object $\mathcal{O}$. Define an exact functor
$F:\mathsf{Proj}(\mathsf{Mod}_{fg}(\mathcal{O}))\rightarrow\mathcal{E}%
\mathsf{LCA}_{\mathcal{O},cg}$, where the latter is the exact category of
exact sequences, by its action on the generator:%
\[
\mathcal{O}\mapsto(\mathcal{O}\hookrightarrow\bigoplus_{\sigma\in S}%
\mathbb{R}_{\sigma}\twoheadrightarrow\mathbb{T}_{\mathcal{O}})\text{.}%
\]
That is: We send $\mathcal{O}$ to the Minkowski sequence of $\mathcal{O}$, cf.
Equation \ref{lMinkSeqForIdeals}. If%
\[
p_{i}:\mathsf{Proj}(\mathsf{Mod}_{fg}(\mathcal{O}))\rightarrow\mathsf{LCA}%
_{\mathcal{O},cg}%
\]
for $i=0,1,2$ denote the projections to left, middle and right term, the
additivity of $K$ implies that $p_{1}=p_{0}+p_{2}$ for the induced maps on $K
$. However, $p_{0}$ induces the map $g$ in Diagram \ref{lda2}, while $p_{1} $
is the functor of our claim. Finally, the functor $p_{2}$ can be factored over
the category of compact modules,%
\[
p_{2}:\mathsf{Proj}(\mathsf{Mod}_{fg}(\mathcal{O}))\longrightarrow
\mathsf{LCA}_{\mathcal{O},C}\longrightarrow\mathsf{LCA}_{\mathcal{O}%
,cg}\text{.}%
\]
Induced on $K$, this is $K(\mathsf{Mod}_{fg}(\mathcal{O}))\rightarrow
0\rightarrow K(\mathsf{LCA}_{\mathcal{O},cg})$ by the Eilenberg swindle, Lemma
\ref{lemma_EilenbergSwindle}. Thus, $p_{2}$ induces the zero map on the level
of $K$.\ It follows that $p_{0}$ and $p_{1}$ induce the same map on $K$.
\end{proof}

In the special case of $\mathcal{O}=\mathbb{Z}$, we recover the statement of
Clausen's theorem \cite[Theorem 3.4]{clausen}:

\begin{theorem}
[Clausen]\label{thm_clausen}For every localizing invariant
$K:\operatorname*{Cat}_{\infty}^{\operatorname*{ex}}\rightarrow\mathsf{A}$,
there is a canonical fiber sequence%
\[
K(\mathbb{Z})\longrightarrow K(\mathbb{R})\longrightarrow K(\mathsf{LCA}%
)\text{.}%
\]

\end{theorem}

While this is the same claim as in\ Clausen's work, it does not allow us to
conclude that we actually construct the same fiber sequence. This question
remains open.

\section{Regulators}

In this section, $K$ really stands for $K$-theory. Let $F$ be a number field
with $r$ real places and $s$ complex places and $\mathcal{O}$ its ring of
integers. While this paper is about the computation of $K(\mathsf{LCA}%
_{\mathcal{O}})$, the Minkowski map%
\[
K(\mathcal{O})\longrightarrow K(\mathbb{R})^{r}\oplus K(\mathbb{C})^{s}%
\]
has, uncalled for, showed up in this computation.

Although there is a priori not much reason to expect regulators to play a
r\^{o}le in the context of the category $\mathsf{LCA}_{\mathcal{O}}$, it is
noteworthy that the Borel regulator factors over the above map. This suggests
that some \textquotedblleft shadow\textquotedblright\ of the Borel regulator
may be definable on the $K$-theory of $\mathsf{LCA}_{\mathcal{O}}$. This turns
out to be virtually true, but the image of the map becomes the torus which
arises from quotienting out the regulator lattice from the target cohomology
group; so it surjects onto the torus whose natural volume is precisely what is
customarily called the $n$-th Borel regulator value, and which up to a
rational factor is a regularized zeta value,%
\[
\sim_{\mathbb{Q}^{\times}}\zeta_{F}^{\star}(1-n)\text{.}%
\]
As a side application of this, we obtain that the sequence in Theorem
\ref{thm_MainThm} splits into short exact sequences rationally.\medskip

Let us briefly recall the concept of the Borel regulator. The Dirichlet
regulator map%
\begin{equation}
\Phi_{\operatorname*{Dirichlet}}:\mathcal{O}^{\times}\oplus\mathbb{Z}%
\longrightarrow\mathbb{R}^{r}\oplus\mathbb{R}^{s} \label{lww2}%
\end{equation}
sends a unit $x\in\mathcal{O}^{\times}$ to the vector $(\log\left\vert
\sigma(x)\right\vert )_{\sigma}$, where $\sigma$ runs through all real places
and one representative of each complex place, and $\mathbb{Z}$ gets mapped to
the right diagonally. The image of $\Phi_{\operatorname*{Dirichlet}}$ is a
full rank lattice and the kernel a finite group, namely the roots of unity in
$F^{\times}$. The covolume of the lattice is known as the Dirichlet regulator
value itself.

In view of $K_{1}(\mathcal{O})\cong\mathcal{O}^{\times}$, one may rephrase
this map as a map defined on $K$-theory. Borel has given a far-reaching
generalization of this map. Suppose $n\geq1$. Conceptually, following
Beilinson's picture, the Borel regulator is a canonical morphism%
\begin{equation}
\Phi_{\operatorname*{Borel}}:K_{2n-1}(\mathcal{O})\longrightarrow
H_{\mathcal{D}}^{1}((\operatorname*{Spec}F)_{/\mathbb{R}},\mathbb{R}%
(n))\text{,} \label{lww1}%
\end{equation}
but in the case at hand the Beilinson--Deligne cohomology group on the right
is just a few copies of the real numbers, specifically%
\[
\Phi_{\operatorname*{Borel}}:K_{2n-1}(\mathcal{O})\longrightarrow\left\{
\begin{array}
[c]{ll}%
\mathbb{R}^{s} & \text{if }n\text{ is even}\\
\mathbb{R}^{r+s} & \text{if }n\text{ is odd.}%
\end{array}
\right.
\]
We write $\Phi_{\operatorname*{Borel}}:K_{2n-1}(\mathcal{O})\rightarrow
\mathbb{R}^{d_{n}}$ with $d_{n}$ defined accordingly for simplicity.

\begin{theorem}
[Borel]\label{thm_BorelReg}For $n\geq2$, the image of $\Phi
_{\operatorname*{Borel}}$ is a full rank $\mathbb{Z}$-lattice and $\ker
(\Phi_{\operatorname*{Borel}})$ is a finite group. The covolume of the lattice
satisfies%
\[
u_{n}\cdot\lim_{z\rightarrow1-n}(z-1+n)^{-d_{n}}\cdot\zeta_{F}(z)\text{,}%
\]
where $u_{n}$ is a non-zero rational number.
\end{theorem}

See \cite{MR944994}, \cite{MR1869655} for details.

\begin{remark}
Instead, for $n=1$, one can use Dirichlet's analytic class number formula,
giving an analogous statement. Note the slight difference in that Dirichlet
regulator as in Equation \ref{lww2} needs the additional artificial summand
$\mathbb{Z}$ to have a full rank lattice image. This has no counterpart for
Borel's regulator. There are some other subtle differences in degree one. For
example $K_{1}(\mathcal{O})=\mathcal{O}^{\times}$ and $K_{1}(F)=F^{\times}$
are very different, while $K_{2n-1}(\mathcal{O})\cong K_{2n-1}(F)$ for all
$n\geq2$.
\end{remark}

\begin{remark}
For $n\geq1$, the $K$-groups $K_{2n}(\mathcal{O})$ are finite, so we cannot
have non-trivial maps to the reals and thus not expect anything like a regulator.
\end{remark}

We are now ready to prove the following:

\begin{theorem}
\label{thm_RationalSplitting}Let $F$ be a number field and $\mathcal{O}$ its
ring of integers. With rational coefficients, the long exact sequence of
Theorem \ref{thm_MainThm} splits into short exact sequences%
\[
0\rightarrow K_{n}(\mathcal{O})_{\mathbb{Q}}\rightarrow K_{n}(\mathbb{R}%
)_{\mathbb{Q}}^{r}\oplus K_{n}(\mathbb{C})_{\mathbb{Q}}^{s}\rightarrow
K_{n}(\mathsf{LCA}_{\mathcal{O}})_{\mathbb{Q}}\rightarrow0
\]
for all $n$.
\end{theorem}

Note that for $n\geq2$ and $n$ even, $K_{n}(\mathcal{O})_{\mathbb{Q}}=0$.

\begin{proof}
We apply the fiber sequence of Theorem \ref{thm_MainThm} for non-connective
$K$-theory $\mathbb{K}$. As the ring $\mathcal{O}$ and fields (here
$\mathbb{R}$ and $\mathbb{C}$) are regular, connective and non-connective
$K$-theory agree for the first two terms in the fiber sequence. Thus, the
resulting long exact sequence also holds for connective $K$-theory. We display
it below in the bottom row:%
\begin{equation}%
%%
%& & \mathbb{R}^{d_{2n-1}} \\
%\cdots\ar[r]^-{{\partial}_{2n}} & {K}_{2n-1}(\mathcal{O}) \ar[r]^-{\mu_{2n-1}}
%\ar[ur]^{\Phi_{\operatorname{Borel}}} &
%{K}_{2n-1}(\mathbb{R})^{r}\oplus{K}_{2n-1}(\mathbb{C})^{s} \ar[u]_{\phi}
%\ar[r] &
%{K}_{2n-1}(\mathsf{LCA}_{\mathcal{O}}) \ar[r]^-{{\partial}_{2n-1}} & \cdots}}
%}%
%%
\xymatrix@!R=0.52in{
& & \mathbb{R}^{d_{2n-1}} \\
\cdots\ar[r]^-{{\partial}_{2n}} & {K}_{2n-1}(\mathcal{O}) \ar[r]^-{\mu_{2n-1}}
\ar[ur]^{\Phi_{\operatorname{Borel}}} &
{K}_{2n-1}(\mathbb{R})^{r}\oplus{K}_{2n-1}(\mathbb{C})^{s} \ar[u]_{\phi}
\ar[r] &
{K}_{2n-1}(\mathsf{LCA}_{\mathcal{O}}) \ar[r]^-{{\partial}_{2n-1}} & \cdots}
\label{lww4}%
\end{equation}
Next, note that for $n\geq2$ the Borel regulator is defined as the composition
of two maps: Firstly, a Minkowski type map%
\begin{equation}
K_{2n-1}(\mathcal{O})\overset{\sim}{\longrightarrow}K_{2n-1}(F)\overset
{M}{\longrightarrow}K_{2n-1}(\mathbb{R})^{r}\oplus K_{2n-1}(\mathbb{C}%
)^{s}\longrightarrow\left[  \bigoplus_{\sigma}K_{2n-1}(\mathbb{C})\right]
^{G}\text{,} \label{lww3}%
\end{equation}
where $\sigma$ runs through \textit{all} complex embeddings, and $[-]^{G}$
denotes taking $G$-fixed points under the complex conjugation involution, $M$
is induced from the Minkowski embedding as a ring homomorphism. The left-most
map is induced from the inclusion $\mathcal{O}\hookrightarrow F$; this is an
isomorphism on $K$-theory \cite[Ch. V, \S 6, Theorem 6.8]{MR3076731}.
Secondly, the regulator map defined on the level of $\mathbb{C}$, namely%
\[
K_{2n-1}(\mathbb{C})\overset{\operatorname*{Hur}}{\longrightarrow}%
H_{2n-1}(\left(  B\operatorname{GL}(\mathbb{C})_{\delta}\right)
^{+},\mathbb{Z})\overset{\sim}{\longrightarrow}H_{2n-1}(B\operatorname{GL}%
(F)_{\delta},\mathbb{Z})\longrightarrow\mathbb{R}^{d_{n}}\text{,}%
\]
where the first map is the Hurewicz morphism, originating from Quillen's model
for $K$-theory via the plus construction, $B\operatorname{GL}(\mathbb{C}%
)_{\delta}$ denotes the classifying space of $\operatorname{GL}(\mathbb{C}%
)_{\delta}$ as a group with the discrete topology; the second map comes from
the invariance of group homology under the plus construction, and the last
arrow is the pairing with the Borel class (this is induced from the continuous
group cohomology of $\operatorname{GL}(\mathbb{C})$, pulled back along the
tautologically continuous map $\operatorname{GL}(\mathbb{C})_{\delta
}\rightarrow\operatorname{GL}(\mathbb{C})$, see \cite{MR944994},
\cite{MR1869655}). The composition of these maps is $\Phi
_{\operatorname*{Borel}}$ in the above diagram. As we can see from Equation
\ref{lww3}, up until the Minkowski map, the first ingredients in this
compositions of maps agrees with the map appearing in the bottom sequence in
Figure \ref{lww4}. Thus, if we denote the remaining steps by $\phi$, we get
the middle upward arrow in the diagram. Next, note that%
\begin{equation}
\ker(\mu_{2n-1})\subseteq\ker(\Phi_{\operatorname*{Borel}}) \label{lww6}%
\end{equation}
and since the kernel of the Borel regulator is finite (cf. Theorem
\ref{thm_BorelReg}), so must be $\ker(\mu_{2n-1})$, and thus the image of
$\partial_{2n}$; for all $n\geq2$. Moreover, $K_{2n}(\mathcal{O})$ is finite
for all $n\geq1$, so the image of $\partial_{2n-1}$ is also finite. It follows
that the bottom long exact sequence in Figure \ref{lww4} splits into separate
short exact sequences rationally. This is also true for small $n$: The group
$K_{2}(\mathcal{O})$ is finite, so it remains to check
\[
K_{2}(\mathsf{LCA}_{\mathcal{O}})\rightarrow K_{1}(\mathcal{O})\overset
{\mu_{1}}{\rightarrow}K_{1}(\mathbb{R})^{r}\oplus K_{1}(\mathbb{C}%
)^{s}\rightarrow K_{1}(\mathsf{LCA}_{\mathcal{O}})\rightarrow K_{0}%
(\mathcal{O})\text{,}%
\]
but $\ker(\mu_{1})$ is just the roots of unity, which is a finite group, so
$\partial_{2}$ has finite image. Correspondingly,%
\[
K_{0}(\mathcal{O})\cong\mathbb{Z}\oplus\operatorname*{Cl}(\mathcal{O}%
)\overset{\mu_{0}}{\longrightarrow}\mathbb{Z}^{r}\oplus\mathbb{Z}^{s}\text{,}%
\]
so again $\mu_{0}$ has finite kernel by the finiteness of the class group. It
follows that the entire long exact sequence splits into short exact sequences rationally.
\end{proof}

\begin{theorem}
Let $F$ be a number field and $\mathcal{O}$ its ring of integers. Suppose
$n\geq2$. There is a canoncial Borel regulator map on the torus quotient by
the regulator lattice%
\[
K_{2n-1}^{\#}(\mathsf{LCA}_{\mathcal{O}})\twoheadrightarrow\mathbb{R}^{d_{n}%
}/\Phi_{\operatorname*{Borel}}(K_{2n-1}(\mathcal{O}))\text{,}%
\]
where $K_{2n-1}^{\#}(\mathsf{LCA}_{\mathcal{O}})$ is some well-defined finite
index subgroup of $K_{2n-1}(\mathsf{LCA}_{\mathcal{O}})$.
\end{theorem}

\begin{proof}
We use the same setup as in the previous proof; with the same notation. Take
Figure \ref{lww4}. We make $\mu_{2n-1}$ injective by quotienting out $\ker
(\mu_{2n-1})$. By Equation \ref{lww6} the Borel regulator is still
well-defined on this quotient. Secondly, we replace $K_{2n-1}(\mathsf{LCA}%
_{\mathcal{O}})$ by the image of the incoming arrow. This agrees with
$\ker(\partial_{2n-1})$, as in%
\[
\ker(\partial_{2n-1})\hookrightarrow K_{2n-1}(\mathsf{LCA}_{\mathcal{O}%
})\twoheadrightarrow\text{image in }K_{2n-2}(\mathcal{O})
\]
and since $K_{2n-2}(\mathcal{O})$ is a finite group, it follows that this is a
finite index subgroup. Now, the universal property of cokernels implies that
$\phi$ descends to a morphism%
\[
\overline{\phi}:\ker(\partial_{2n-1})\longrightarrow\mathbb{R}^{d_{n}}%
/\Phi_{\operatorname*{Borel}}(K_{2n-1}(\mathcal{O}))
\]
and as $\Phi_{\operatorname*{Borel}}$ has full rank image, and thus is
surjective after tensoring with the reals, $\phi$ and thus $\overline{\phi}$
are surjective.
\end{proof}

%%%%%
%%%%%

\begin{example}(Classical algebraic number theory)
In low degrees, Theorem \ref{thm_MainThm} yields the exact sequence%
\begin{equation}
\cdots \rightarrow K_{1}(\mathcal{O})\rightarrow
K_{1}(\mathbb{R})^{r}\oplus K_{1}(\mathbb{C})^{s}\rightarrow K_{1}(\mathsf{%
LCA}_{\mathcal{O}})\rightarrow K_{0}(\mathcal{O})\rightarrow \mathbb{Z}%
^{r+s}\rightarrow K_{0}(\mathsf{LCA}_{\mathcal{O}})\rightarrow 0\text{.}
\label{lj1}
\end{equation}%
Note that all the \textquotedblleft usual suspects\textquotedblright\ of
basic algebraic number theory show up:\ (1) We have $K_{0}(\mathcal{O})\cong 
\mathbb{Z}\oplus \operatorname{Cl}(\mathcal{O})$. Since torsion must map to zero
in $\mathbb{Z}^{r+s}$, yet the free rank part maps diagonally to the real
and complex vector spaces, the last three terms split off to form%
\begin{equation*}
K_{0}(\mathsf{LCA}_{\mathcal{O}})\cong \mathbb{Z}^{r+s-1}\text{.}
\end{equation*}%
Thus, $K_{0}(\mathsf{LCA}_{\mathcal{O}})$ has the same free rank as the unit
group $\mathcal{O}^{\times }$. Moreover, using the Bass--Milnor--Serre
Theorem the middle part of the sequence becomes%
\begin{equation*}
K_{2}(\mathsf{LCA}_{\mathcal{O}})\overset{\nu }{\rightarrow }\mathcal{O}%
^{\times }\overset{\mu }{\rightarrow }(\mathbb{R}^{\times })^{r}\oplus (%
\mathbb{C}^{\times })^{s}\rightarrow K_{1}(\mathsf{LCA}_{\mathcal{O}%
})\rightarrow \operatorname{Cl}(\mathcal{O})\rightarrow 0
\end{equation*}%
and $\mu $ sends a unit $\alpha \in \mathcal{O}^{\times }$ to its embedding $%
\sigma (\alpha )$ in the real or complex numbers. It follows that $\nu =0$.
Composing the middle term with a logarithm,%
\begin{eqnarray*}
(\mathbb{R}^{\times })^{r}\oplus (\mathbb{C}^{\times })^{s} &\longrightarrow
&\mathbb{R}^{r+s} \\
(\ldots ,x_{i},\ldots ) &\longmapsto &(\ldots ,\log \left\vert \sigma
(x_{i})\right\vert ,\ldots )\text{,}
\end{eqnarray*}%
we thus find both the Dirichlet embedding as well as the class group appear
in Equation \ref{lj1}.
\end{example}

%%%%%

\section{Comparison with a Tate category}

For background on Tate categories we refer to \cite{MR2872533},
\cite{MR3510209}. Write $\mathsf{Mod}_{fin}(\mathcal{O})$ for the abelian
category of finitely generated Artinian $\mathcal{O}$-modules; or equivalently
of $\mathcal{O}$-modules whose underlying set is finite.

\begin{theorem}
Let $K:\operatorname*{Cat}_{\infty}^{\operatorname*{ex}}\rightarrow\mathsf{A}$
be a localizing invariant. There is a canonical map of fiber sequences in
$\mathsf{A}$,%
\[%
%%
%K({\mathsf{Mod}_{fin}(\mathcal{O})}) \ar[r] \ar[d]_{\sigma} & 0 \ar
%[r] \ar[d] & K({\mathsf{Tate}({{\mathsf{Mod}_{fin}(\mathcal{O})}})}%
%) \ar[d]^{\gamma} \\
%K(\mathcal{O}) \ar[r] & K(\mathbb{R})^r \oplus K(\mathbb{C})^s \ar
%[r] & K(\mathsf{LCA}_{\mathcal{O}}),
%}}}%
%%
\xymatrix{
K({\mathsf{Mod}_{fin}(\mathcal{O})}) \ar[r] \ar[d]_{\sigma} & 0 \ar
[r] \ar[d] & K({\mathsf{Tate}({{\mathsf{Mod}_{fin}(\mathcal{O})}})}%
) \ar[d]^{\gamma} \\
K(\mathcal{O}) \ar[r] & K(\mathbb{R})^r \oplus K(\mathbb{C})^s \ar
[r] & K(\mathsf{LCA}_{\mathcal{O}}),
}%
\]
where the bottom row is the sequence of Theorem \ref{thm_MainThm}, $\sigma$ is
induced from the inclusion functor $\mathsf{Mod}_{fin}(\mathcal{O})$
$\hookrightarrow\mathsf{Mod}_{fg}(\mathcal{O})$, and $\gamma$ is induced from
the evaluation of the ind-pro limit inside $\mathsf{LCA}_{\mathcal{O}}$.
\end{theorem}

\begin{proof}
We return to Diagram \ref{ld5}:%
\begin{equation}%
%%
%{\mathsf{Mod}_{fg}(\mathcal{O})} \ar@{^{(}->}[r] \ar[d] & {\mathsf
%{Mod}(\mathcal{O})} \ar@{->>}[r] \ar[d] & {\mathsf{Mod}(\mathcal{O})}%
%/{\mathsf{Mod}_{fg}(\mathcal{O})} \ar[d] \\
%\mathsf{LCA}_{\mathcal{O},cg} \ar@{^{(}->}[r] & \mathsf{LCA}_{\mathcal{O}}
%\ar@{->>}[r] & {\mathsf{LCA}_{\mathcal{O}}}/{\mathsf{LCA}_{\mathcal{O},cg}}.
%}} }%
%%
\xymatrix{
{\mathsf{Mod}_{fg}(\mathcal{O})} \ar@{^{(}->}[r] \ar[d] & {\mathsf
{Mod}(\mathcal{O})} \ar@{->>}[r] \ar[d] & {\mathsf{Mod}(\mathcal{O})}%
/{\mathsf{Mod}_{fg}(\mathcal{O})} \ar[d] \\
\mathsf{LCA}_{\mathcal{O},cg} \ar@{^{(}->}[r] & \mathsf{LCA}_{\mathcal{O}}
\ar@{->>}[r] & {\mathsf{LCA}_{\mathcal{O}}}/{\mathsf{LCA}_{\mathcal{O},cg}}.
}
\label{lciv1}%
\end{equation}
with the exact equivalence of Lemma \ref{lemma_equivq} on the right. There is
a quite analogous diagram, first constructed by Sho Saito \cite[end of page
$9$]{MR3317759},%
\begin{equation}%
%%
%{\mathsf{Mod}_{fin}(\mathcal{O})} \ar@{^{(}->}[r] \ar[d] & \mathsf{Ind}%
%^{a}({\mathsf{Mod}_{fin}(\mathcal{O})}) \ar@{->>}[r] \ar[d] & {\mathsf
%{Ind}^{a}({\mathsf{Mod}_{fin}(\mathcal{O})})}/{{\mathsf{Mod}_{fin}(\mathcal
%{O})}}
%\ar[d]^{\sim} \\
%\mathsf{Pro}^{a}({\mathsf{Mod}_{fin}(\mathcal{O})}) \ar@{^{(}->}%
%[r] & \mathsf{Tate}({\mathsf{Mod}_{fin}(\mathcal{O})}) \ar@{->>}%
%[r] & {\mathsf{Tate}({\mathsf{Mod}_{fin}(\mathcal{O})})}/{\mathsf{Pro}%
%^{a}({\mathsf{Mod}_{fin}(\mathcal{O})})},
%}} }%
%%
\xymatrix{
{\mathsf{Mod}_{fin}(\mathcal{O})} \ar@{^{(}->}[r] \ar[d] & \mathsf{Ind}%
^{a}({\mathsf{Mod}_{fin}(\mathcal{O})}) \ar@{->>}[r] \ar[d] & {\mathsf
{Ind}^{a}({\mathsf{Mod}_{fin}(\mathcal{O})})}/{{\mathsf{Mod}_{fin}(\mathcal
{O})}}
\ar[d]^{\sim} \\
\mathsf{Pro}^{a}({\mathsf{Mod}_{fin}(\mathcal{O})}) \ar@{^{(}->}%
[r] & \mathsf{Tate}({\mathsf{Mod}_{fin}(\mathcal{O})}) \ar@{->>}%
[r] & {\mathsf{Tate}({\mathsf{Mod}_{fin}(\mathcal{O})})}/{\mathsf{Pro}%
^{a}({\mathsf{Mod}_{fin}(\mathcal{O})})},
}
\label{lciv2}%
\end{equation}
where both rows are exact sequences of exact categories. The right downward
arrow is an exact equivalence (\cite[Lemma 3.3]{MR3317759}, loc. cit. this is
done for countable diagrams, but it works just as well without a cardinality
constraint, see \cite[Prop. 5.32]{MR3510209}). Next, we note that we can find
exact functors mapping the second diagram commutatively to the former: On the
upper left we just the inclusion $\mathsf{Mod}_{fin}(\mathcal{O}%
)\subset\mathsf{Mod}_{fg}(\mathcal{O})$. Further, there are exact functors of
evaluation%
\[
\alpha:\mathsf{Ind}^{a}(\mathsf{Mod}_{fin}(\mathcal{O}))\longrightarrow
\mathsf{Mod}(\mathcal{O})\qquad\text{and}\qquad\beta:\mathsf{Pro}%
^{a}(\mathsf{Mod}_{fin}(\mathcal{O}))\longrightarrow\mathsf{LCA}%
_{\mathcal{O},C}\text{.}%
\]
(we justify this: The category $\mathsf{Mod}(\mathcal{O})$ is Grothendieck
abelian. Thus, all colimits exist and the evaluation of a colimit is an exact
functor. This settles $\alpha$. Note that $\beta$ unravels to be a functor
$\beta:\mathsf{Ind}^{a}(\mathsf{Mod}_{fin}(\mathcal{O})^{op})^{op}%
\rightarrow\mathsf{LCA}_{\mathcal{O},C}$, and thus setting it up is equivalent
to giving an exact functor $\beta^{op}:\mathsf{Ind}^{a}(\mathsf{Mod}%
_{fin}(\mathcal{O})^{op})\rightarrow\mathsf{LCA}_{\mathcal{O},D} $. The
category $\mathsf{Mod}_{fin}(\mathcal{O})$ is equivalent to its opposite
through Pontryagin duality, and $\mathsf{LCA}_{\mathcal{O},D}\cong
\mathsf{Mod}(\mathcal{O})$, so the same argument as for $\alpha$ settles that
$\beta$ exists and is exact. Finally, $\mathsf{LCA}_{\mathcal{O},C}$ is a full
subcategory of $\mathsf{LCA}_{\mathcal{O},cg}$). Moreover, there is a functor
of evaluation%
\[
\gamma:\mathsf{Tate}^{el}(\mathsf{Mod}_{fin}(\mathcal{O}))\longrightarrow
\mathsf{LCA}_{\mathcal{O}}\text{,}%
\]
defined analogously: We use the presentation of $\mathsf{Tate}^{el}%
(\mathsf{Mod}_{fin}(\mathcal{O}))$ as a full subcategory of $\mathsf{Pro}%
^{a}\mathsf{Ind}^{a}(\mathsf{Mod}_{fin}(\mathcal{O}))$. Every object can be
presented as an admissible Pro diagram $X:I\rightarrow\mathsf{Ind}%
^{a}(\mathsf{Mod}_{fin}(\mathcal{O}))$ such that for all arrows in the diagram
category $I$, we have $X_{i,j}\hookrightarrow X(i)\twoheadrightarrow X(j)$
with $X_{i,j}$ an Artinian $\mathcal{O}$-module. The $X(-)$ are discrete
groups. We note that $X(i)\twoheadrightarrow X(j)$ is a proper map in the
sense of topology, i.e. the preimage of a compact set is compact. This is true
since the compact subsets are precisely the finite ones and since the kernel
$X_{i,j}$ is finite, the preimage of a finite set is still a finite set, thus
compact. It follows that the limit is locally compact, see for example
\cite[Corollary 1.26 (a)]{MR2337107} or \cite[(38.1)\ Lemma]{MR2226087}. As
for $\alpha,\beta$ one finds that $\gamma$ is exact. Finally, as
$\mathsf{LCA}_{\mathcal{O}}$ is idempotent complete, $\gamma$ canonically
extends to the idempotent completions, i.e. $\mathsf{Tate}(\mathsf{Mod}%
_{fin}(\mathcal{O}))$ in the case at hand, \cite[Prop. 6.10]{MR2606234}. The
functors on the quotient categories on the right in Diagrams \ref{lciv1} and
\ref{lciv2} are induced from these functors and the commutativity of the
square on the left. As a result of this, following the construction in the
proof of Theorem \ref{thm_clausen}, we obtain a canonical morphism of fiber
sequences%
\[%
%%
%K({\mathsf{Mod}_{fin}(\mathcal{O})}) \ar[r] \ar[d]_{\sigma} & 0 \ar
%[r] \ar[d] & K({\mathsf{Tate}({{\mathsf{Mod}_{fin}(\mathcal{O})}})}%
%) \ar[d]^{\gamma} \\
%K(\mathcal{O}) \ar[r] & K(\mathbb{R})^r \oplus K(\mathbb{C})^s \ar
%[r] & K(\mathsf{LCA}_{\mathcal{O}}),
%}}}%
%%
\xymatrix{
K({\mathsf{Mod}_{fin}(\mathcal{O})}) \ar[r] \ar[d]_{\sigma} & 0 \ar
[r] \ar[d] & K({\mathsf{Tate}({{\mathsf{Mod}_{fin}(\mathcal{O})}})}%
) \ar[d]^{\gamma} \\
K(\mathcal{O}) \ar[r] & K(\mathbb{R})^r \oplus K(\mathbb{C})^s \ar
[r] & K(\mathsf{LCA}_{\mathcal{O}}),
}%
\]
We have used the\ Eilenberg swindle to get $K(\mathsf{Pro}^{a}(\mathsf{Mod}%
_{fin}(\mathcal{O})))=0$. This proves our claim.
\end{proof}

\begin{example}
There is also a variant of Diagram \ref{lciv2}, replacing $\mathsf{Mod}_{fin}
$ by $\mathsf{Mod}_{fg}$, but in order to map it to LCA groups, we would need
functors $\mathsf{Ind}^{a}(\mathsf{Mod}_{fg}(\mathcal{O}))\rightarrow
\mathsf{Mod}(\mathcal{O})$ and $\mathsf{Pro}^{a}(\mathsf{Mod}_{fg}%
(\mathcal{O}))\rightarrow\mathsf{LCA}_{\mathcal{O},cg}$. While the former
exists and is even an equivalence of abelian categories, the latter cannot
reasonably be set up. For example, for $\mathcal{O}=\mathbb{Z}$ the infinite
product $\prod_{\mathbb{Z}}\mathbb{Z}$ fails to be compactly generated. So it
is not possible to construct a functor $\mathsf{Tate}(\mathsf{Mod}%
_{fg}(\mathcal{O}))\rightarrow\mathsf{LCA}_{\mathcal{O}}$ by the method of the
above theorem.
\end{example}

\begin{proposition}
[Soul\'{e}]For non-connective $K$-theory $\mathbb{K}$, the map $\sigma
:\mathbb{K}_{i}(\mathsf{Mod}_{fin}(\mathcal{O}))\rightarrow\mathbb{K}%
_{i}(\mathcal{O})$ of the previous theorem is zero for all integers $i\neq0$.
If $F$ has trivial class group, it is also trivial in degree zero.
\end{proposition}

\begin{proof}
Temporarily write $K$ for connective $K$-theory. Firstly, $\mathsf{Mod}%
_{fin}(\mathcal{O})$ is a Noetherian abelian category, so by Schlichting's
Connectivity Theorem \cite[Theorem 7]{MR2206639} we have $\mathbb{K}%
_{i}(\mathsf{Mod}_{fin}(\mathcal{O}))=0$ for $i<0$; and moreover
$\mathbb{K}_{0}(\mathsf{Mod}_{fin}(\mathcal{O}))=K_{0}(\mathsf{Mod}%
_{fin}(\mathcal{O}))$, \cite[Remark 3]{MR2206639}; and $\mathbb{K}%
_{i}(\mathsf{Mod}_{fin}(\mathcal{O}))=K_{i}(\mathsf{Mod}_{fin}(\mathcal{O}))$
for $i\geq1$ anyway. Finally, $\mathcal{O}$ is a regular ring, so
$\mathbb{K}_{i}(\mathcal{O})=K_{i}(\mathcal{O})$. Thus, it suffices to prove
the claim for connective $K$-theory. For this, we follow the presentation in
Weibel \cite[Chapter V, \S 6]{MR3076731}: For every number field $F$ and ring
of integers $\mathcal{O}$, we have the exact sequence $\mathsf{Mod}%
_{fin}(\mathcal{O})\hookrightarrow\mathsf{Mod}_{fg}(\mathcal{O}%
)\twoheadrightarrow\mathsf{Mod}_{fg}(F)$ of abelian categories, and the
resulting localization sequence reads%
\[
\cdots\longrightarrow K_{i+1}(\mathcal{O})\longrightarrow K_{i+1}%
(F)\overset{\partial_{x}^{F}}{\longrightarrow}K_{i}(\mathsf{Mod}%
_{fin}(\mathcal{O}))\overset{\iota_{i}}{\longrightarrow}K_{i}(\mathcal{O}%
)\longrightarrow\cdots\text{.}%
\]
D\'{e}vissage yields%
\[
K_{i}(\mathsf{Mod}_{fin}(\mathcal{O}))\cong\coprod_{x\in(\operatorname*{Spec}%
\mathcal{O})_{(0)}}K_{i}(\kappa(x))\text{,}%
\]
the $\kappa(x)$ are finite fields, so by Quillen's computation of the
$K$-theory of finite fields we deduce that $\iota_{i}=0$ for $i\geq2$ even
since $K_{i}(\mathbb{F}_{q})=0$ for these $i$. On the other hand, Soul\'{e}
shows that for $i$ odd, $i\geq1$, the sequence splits into short exact
sequences, so that $\coprod\partial_{x}^{F}$ is surjective, cf. \cite[Ch. V,
\S 6, Theorem 6.8]{MR3076731}. Thus, $\iota_{i}=0$ for $i\geq1$ odd. For
$i=0$, we get $F^{\times}\overset{\partial_{x}^{F}}{\longrightarrow}%
\coprod_{x}\mathbb{Z}\overset{\iota_{0}}{\longrightarrow}K_{0}(\mathcal{O}%
)\overset{\alpha}{\longrightarrow}K_{0}(F)\longrightarrow0$ and $K_{0}%
(\mathcal{O})\cong\mathbb{Z}\oplus\operatorname*{Cl}(\mathcal{O})$. Thus, as
$\alpha$ is the rank map, it follows that $\iota_{0}$ is a surjection on the
class group of $F$.
\end{proof}

\begin{remark}
Thus, we obtain a lift%
\[
\mathbb{K}_{i}(\mathsf{Tate}(\mathsf{Mod}_{fin}(\mathbb{Z})))\longrightarrow
\mathbb{K}_{i}(\mathbb{R})/\operatorname*{im}(\mathbb{K}_{i}(\mathbb{Z}%
))\text{.}%
\]

\end{remark}%

\appendix

\section{\label{sect_AppendixInjectives}Characterization of injectives and
projectives}

Next, we need to recall the structure of the injective and projective objects
in $\mathsf{LCA}_{\mathcal{O}}$.\ See \cite[\S 11]{MR2606234} for a discussion
of these concepts in exact categories. For $\mathsf{LCA}$ these results are
due to Moskowitz \cite{MR0215016}, and for $\mathsf{LCA}_{\mathcal{O}}$ due to
Kryuchkov \cite{MR1620000}. Unfortunately, Kryuchkov's results appear to be
unknown to a wider audience. The paper was published over 20 years ago and
seems not to have been cited in any research article since then. We miss out
on some interesting results. Below, we give an exposition of how Kryuchkov's
more general results can be obtained if we take the ones for $\mathsf{LCA}$ in
\cite{MR0215016} for granted. We claim no originality for this whatsoever; the
strategy follows \cite{MR0215016} and \cite{MR1620000}.

\begin{lemma}
\label{lemma_InjectivesAreConnected}Suppose $I$ is an injective object in
$\mathsf{LCA}_{\mathcal{O}}$. Then $I$ is connected and injective as an
algebraic $\mathcal{O}$-module.
\end{lemma}

\begin{proof}
Connectedness: Let $y\in I$ be arbitrary. Consider the Minkowski sequence for
$\mathcal{O}$ itself,\ Equation \ref{lMinkSeqForIdeals},%
\[%
%%
%0 \ar[r] & \mathcal{O} \ar[r] \ar[d]_{g} & \bigoplus_{\sigma}\mathbb
%{R}_{\sigma} \ar[r] \ar@{-->}[dl]^{\tilde{g}} & \mathbb{T}_{\mathcal{O} }
%\ar[r] & 0\\
%& I,
%}}}%
%%
\xymatrix{
0 \ar[r] & \mathcal{O} \ar[r] \ar[d]_{g} & \bigoplus_{\sigma}\mathbb
{R}_{\sigma} \ar[r] \ar@{-->}[dl]^{\tilde{g}} & \mathbb{T}_{\mathcal{O} }
\ar[r] & 0\\
& I,
}%
\]
where $g(r):=ry$. As $I$ is injective, $\tilde{g}$ exists. However,
$\bigoplus_{\sigma}\mathbb{R}_{\sigma}$ is connected and $\tilde{g}$
continuous, so the image of $\tilde{g}$ lies in the connected component of
zero in $I$. However, $y$ lies in this image and $y$ was arbitrary, so this
connected component is the entire group. Injectivity: Similarly, consider the
sequence $\mathcal{O}\overset{\cdot a}{\hookrightarrow}\mathcal{O}%
\twoheadrightarrow\mathcal{O}/a$ with $g:\mathcal{O}\rightarrow I$ as before.
As $I$ is injective, we get an analogous lift $\tilde{g}:\mathcal{O}%
\rightarrow I$ so that $y=g(1)=\tilde{g}(a)=a\tilde{g}(1)$, showing that $y$
is a multiple of $a$. As $y$ was arbitrary, it follows that $I$ is divisible
as an $\mathcal{O}$-module and since $\mathcal{O}$ is Dedekind, it is an
injective module.
\end{proof}

\begin{lemma}
\label{lemma_InjectiveObjectsAreDetectedOnCG}Let $I\in\mathsf{LCA}%
_{\mathcal{O}}$ be an object such that for every diagram given by the solid
arrows in%
\[%
%%
%0 \ar[r] & H \ar[r] \ar[d]_{g} & G \ar[r] \ar@{-->}[dl]^{\tilde{g}}
%& G/H \ar[r] & 0\\
%& I,
%}}}%
%%
\xymatrix{
0 \ar[r] & H \ar[r] \ar[d]_{g} & G \ar[r] \ar@{-->}[dl]^{\tilde{g}}
& G/H \ar[r] & 0\\
& I,
}%
\]
and $G\in\mathsf{LCA}_{\mathcal{O},cg}$, the dashed lift $\tilde{g}$ exists.
Then $I$ is an injective object in $\mathsf{LCA}_{\mathcal{O}}$. That is:
Being injective can be detected by being injective for compactly generated modules.
\end{lemma}

\begin{proof}
We sketch the argument for completeness, but the argument is exactly the same
as used by Moskowitz and also Kryuchkov: Firstly, we claim that $I$ is
injective as an algebraic $\mathcal{O}$-module, without topology. For this
copy the argument in the proof of Lemma \ref{lemma_InjectivesAreConnected},
noting that it only uses injectivity for exact sequences whose middle term is
compactly generated, so by our assumptions the same argument is valid in this
case. Secondly, pick a compactly generated clopen $\mathcal{O}$-submodule $L$,
which is possible by Lemma \ref{lemma_ExistClopenCGSubmodule}. Then $L\cap
H\hookrightarrow L\twoheadrightarrow L/(L\cap H)$ is an exact sequence in
$\mathsf{LCA}_{\mathcal{O}}$. By assumption the restriction of $f $ to $L\cap
H$, call it $f_{1}:L\cap H\rightarrow I$, has a lift $\tilde{f}_{1}%
:L\rightarrow I$ in $\mathsf{LCA}_{\mathcal{O}}$. Now consider the exact
sequence $L\cap H\hookrightarrow L\oplus H\twoheadrightarrow L+H$ in
$\mathsf{LCA}_{\mathcal{O}}$ with $x\mapsto(x,-x)$ and the sum map as arrows.
Define $f_{2}:L\oplus H\rightarrow I$ by $\tilde{f}_{1}+f$. As this map
restricts to the zero map on $L\cap H$, we get a (continuous) factorization%
\begin{equation}
f_{3}:L+H\longrightarrow I\qquad\text{with}\qquad f_{3}(l+h)=\tilde{f}%
_{1}(l)+f(h) \label{lcef1}%
\end{equation}
for any sum decomposition $l+h$ with $l\in L$ and $h\in H$. We then get a
further exact sequence $L+H\hookrightarrow G\twoheadrightarrow D$ with $D$
discrete. On the level of algebraic $\mathcal{O}$-modules, as $I$ is
injective, there exists an algebraic lift $\tilde{f}_{3}:G\rightarrow H$
extending $f_{3}$. As $L$ was open in $G$, so is $L+H$. Thus, checking
continuity of $\tilde{f}_{3}$ amounts to checking it on $L+H$ as $G/(L+H)$ is
discrete. However, on $L+H$ it agrees with $f_{3}$, which we know is
continuous. Thus, $\tilde{f}_{3}$ is a morphism in $\mathsf{LCA}_{\mathcal{O}%
}$. By the description of Equation \ref{lcef1} with $l=0$ it follows that
$\tilde{f}_{3}\mid_{H}=f$, so $\tilde{f}_{3}$ is the desired lift.
\end{proof}

\begin{lemma}
\label{lemma_RSigmaAreInjectives}Let $\nu$ a real or complex place. For every
diagram given by the solid arrows in%
\[%
%%
%0 \ar[r] & G_1 \ar[r] \ar[d]_{g} & G_2 \ar[r] \ar@{-->}[dl]^{\tilde{g}}
%& {G_2}/{G_1} \ar[r] & 0\\
%& \mathbb{R}_{\nu},
%}}}%
%%
\xymatrix{
0 \ar[r] & G_1 \ar[r] \ar[d]_{g} & G_2 \ar[r] \ar@{-->}[dl]^{\tilde{g}}
& {G_2}/{G_1} \ar[r] & 0\\
& \mathbb{R}_{\nu},
}%
\]
and $G_{2}\in\mathsf{LCA}_{\mathcal{O},cg}$, a lift $\tilde{g}$ as indicated
by the dashed arrow exists.
\end{lemma}

\begin{proof}
\textit{(Preparation) }Note that closed subgroups of compactly generated LCA
groups are also compactly generated, \cite[Theorem 2.6]{MR0215016}. Thus,
$G_{1}\in\mathsf{LCA}_{\mathcal{O}}$. We apply the structure theorem
(Proposition \ref{prop_MoskClassifNSSandCGGroups}) to both, writing the
compact summand as a subobject, resulting in the solid arrows of the following
diagram%
\[%
%%
%C_1 \ar@{^{(}->}[d] \ar@{-->}[r]^{\beta} & C_2 \ar@{^{(}->}[d] \\
%G_1 \ar@{^{(}->}[r] \ar@{->>}[d] & G_2 \ar@{->>}[d] \\
%W_1 \ar@{-->}[r]^{\alpha} & W_2,
%}}}%
%%
\xymatrix{
C_1 \ar@{^{(}->}[d] \ar@{-->}[r]^{\beta} & C_2 \ar@{^{(}->}[d] \\
G_1 \ar@{^{(}->}[r] \ar@{->>}[d] & G_2 \ar@{->>}[d] \\
W_1 \ar@{-->}[r]^{\alpha} & W_2,
}%
\]
where $W_{i}\simeq\bigoplus_{\sigma\in I_{i}}\mathbb{R}_{\sigma}%
\oplus\bigoplus_{J\in\mathcal{I}_{i}}J$ for $i=1,2$. Note that $W_{i}$ has no
non-trivial compact subgroups (this is classical for $\mathbb{R}$, and each
$J$ is a free $\mathbb{Z}$-module). Thus the diagonal map $C_{1}\rightarrow
W_{2}$ must be the zero map. Thus, $\beta$ exists by universal property, and
thus $\alpha$, giving the dashed arrows. Moreover, $\beta$ is an admissible
monic by \cite[Cor. 7.7]{MR2606234} (applied to the opposite category).
Moreover, the kernel of $\alpha$ is $W_{1}\cap C_{2}$ (using that $W_{1}$ is
also a subobject), and since $C_{2}$ is compact, this must be the zero object.
This does not yet suffice to know that $\alpha$ is an admissible monic: To
this end, use that $W_{1}$ is a subobject, so $W_{1}\hookrightarrow G_{1}$ is
a closed map, $G_{1}\hookrightarrow G_{2}$ is a closed map, and $G_{2}%
\twoheadrightarrow W_{2}$ is also closed since $C_{2} $ is compact (this
really needs compactness; see\ \cite[Prop. 11]{MR0442141} for a proof). Hence,
the composition $\alpha:W_{1}\rightarrow W_{2}$ is closed. Thus, $\alpha$ is
an admissible monic. From here, we proceed by a step of reductions:\newline%
\textit{(Special Case 1, }$C_{1}=C_{2}=0$\textit{\ and }$W_{2}$\textit{\ is a
vector }$\mathcal{O}$\textit{-module)} In this case $G_{i}=W_{i}$. Consider
the admissible monic $\alpha:W_{1}\hookrightarrow W_{2}$. We can clearly lift
maps along this embedding,%
\[%
%%
%0 \ar[r] & W_1 \ar[r] \ar[d]_{g} & W_2 \ar[r] \ar@{-->}[dl]^{\tilde{g}}
%& {W_2}/{W_1} \ar[r] & 0\\
%& \mathbb{R}_{\nu},
%}}}%
%%
\xymatrix{
0 \ar[r] & W_1 \ar[r] \ar[d]_{g} & W_2 \ar[r] \ar@{-->}[dl]^{\tilde{g}}
& {W_2}/{W_1} \ar[r] & 0\\
& \mathbb{R}_{\nu},
}%
\]
because Proposition \ref{Prop_ClosedSubobjectsOfVectorModules} decomposes it
into convenient direct summands: Since lifting along $\mathbb{R}_{\sigma
}\overset{1}{\rightarrow}\mathbb{R}_{\sigma}$ is trivial, and lifting along
$J\hookrightarrow%
{\textstyle\bigoplus_{\sigma\in I^{\prime}}}
\mathbb{R}_{\sigma}$ just amounts to extending the map defined on $J$ in an
$\mathbb{R}$-linear fashion to the right-hand side, the existence of the lift
is clear. So the construction of the lift $\tilde{g}$ just amounts to scalar
extension.\newline\textit{(Special Case 2, }$C_{1}=C_{2}=0$\textit{)} We have%
\[
W_{2}=\bigoplus_{\sigma\in I_{i}}\mathbb{R}_{\sigma}\oplus\bigoplus
_{J\in\mathcal{I}_{i}}J
\]
and again by the Minkowski embedding as in Equation \ref{lMinkSeqForIdeals} we
may embed $W_{2}$ in a vector $\mathcal{O}$-module. Thanks to \textit{Special
Case 1}, we already know how to lift a map in the case $W_{2}$ is a vector
$\mathcal{O}$-module. Thus, we use this lift and restrict it to $W_{2} $ lying
inside of it.\newline\textit{(General Case)} Consider $g:G_{1}\rightarrow
\mathbb{R}_{\nu}$. Clearly $g(C_{1})=0$ since the image of $C_{1}$ must be
compact, but $\mathbb{R}_{\nu}$ has no non-trivial compact subgroups. Thus,
$g$ factors to $W_{1}\rightarrow\mathbb{R}_{\nu}$. By Special Case 2 we
already know how to lift maps along $W_{1}\hookrightarrow W_{2}$, so we obtain
a lift $\tilde{g}:W_{2}\longrightarrow\mathbb{R}_{\nu}$ and thus are in the
situation%
\[%
%%
%G_1 \ar@{^{(}->}[rr] \ar@{->>}[d] \ar[ddr]^{g} & & G_2 \ar@{->>}[d] \\
%W_1 \ar@{^{(}->}[rr]^{\alpha} \ar@{-->}[dr]_{g} & & W_2 \ar@{-->}%
%[dl]^{\tilde{g}} \\
%& \mathbb{R}_{\nu},
%}}}%
%%
\xymatrix{
G_1 \ar@{^{(}->}[rr] \ar@{->>}[d] \ar[ddr]^{g} & & G_2 \ar@{->>}[d] \\
W_1 \ar@{^{(}->}[rr]^{\alpha} \ar@{-->}[dr]_{g} & & W_2 \ar@{-->}%
[dl]^{\tilde{g}} \\
& \mathbb{R}_{\nu},
}%
\]
producing a lift to $G_{2}$ by composing the maps along the right edge.
\end{proof}

\begin{lemma}
\label{lemma_TJAreInjectives}Let $J$ be a non-zero ideal of $\mathcal{O}$. For
every diagram given by the solid arrows in%
\[%
%%
%0 \ar[r] & G_1 \ar[r] \ar[d]_{g} & G_2 \ar[r] \ar@{-->}[dl]^{\tilde{g}}
%& {G_2}/{G_1} \ar[r] & 0\\
%& \mathbb{T}_{J },
%}}}%
%%
\xymatrix{
0 \ar[r] & G_1 \ar[r] \ar[d]_{g} & G_2 \ar[r] \ar@{-->}[dl]^{\tilde{g}}
& {G_2}/{G_1} \ar[r] & 0\\
& \mathbb{T}_{J },
}%
\]
and $G_{2}\in\mathsf{LCA}_{\mathcal{O},cg}$, a lift $\tilde{g}$ as indicated
by the dashed arrow exists.
\end{lemma}

\begin{proof}
Use Pontryagin duality for $\mathcal{O}$-modules, Theorem
\ref{thm_PontrDualForLCAO}, to obtain the diagram%
\[%
%%
%{G_2}^{\vee} \ar@{->>}[r] & {G_1}^{\vee} \\
%& J \ar[u]_{h} \ar@{-->}[ul]_{\tilde{h}}
%}}}%
%%
\xymatrix{
{G_2}^{\vee} \ar@{->>}[r] & {G_1}^{\vee} \\
& J \ar[u]_{h} \ar@{-->}[ul]_{\tilde{h}}
}%
\]
with $h:=g^{\vee}$ since $\mathbb{T}_{J}^{\vee}\simeq J$. Now, forgetting
topology, the ideal $J$ is a projective algebraic $\mathcal{O}$-module, so the
algebraic $\mathcal{O}$-module lift $\tilde{h}$ exists. Since $J$ is discrete,
it is tautologically continuous. Dualizing again, $\tilde{h}^{\vee}$ is the
required lift.
\end{proof}

Now we are ready to prove the characterization of injectives. For
$\mathsf{LCA}$ this is due to Moskowitz \cite{MR0215016} and the variant for
number fields is due to Kryuchkov \cite{MR1620000}.

\begin{proof}
[Proof of Theorem \ref{thm_InjectivesInLCAO}](1$\Rightarrow$2) Suppose $I$ is
an injective object. By Lemma \ref{lemma_InjectivesAreConnected} it is
connected. Thus, for the additive group we have an isomorphism%
\begin{equation}
I\simeq\mathbb{R}^{n}\oplus C\qquad\text{(in }\mathsf{LCA}\text{)}
\label{lcsis4}%
\end{equation}
for some $n\geq0$ and $C$ compact connected by the classification of connected
LCA\ groups, \cite[Theorem 26]{MR0442141}. For the dual we get $I^{\vee}%
\simeq\mathbb{R}^{n}\oplus D$ with $D$ discrete (in $\mathsf{LCA}$), By Lemma
\ref{lemma_VectorGroupDiscreteInLCALiftsToLCAO} the direct sum decomposition
of Equation \ref{lcsis4} in $\mathsf{LCA}$ gets promoted to an isomorphism
$I^{\vee}\simeq\bigoplus_{\sigma\in I}\mathbb{R}_{\sigma}\oplus D$ in
$\mathsf{LCA}_{\mathcal{O}}$. Hence, dualizing again, $I\cong\bigoplus
_{\sigma\in I}\mathbb{R}_{\sigma}\oplus C$ in $\mathsf{LCA}_{\mathcal{O}}$.
Now as $I$ is an injective object its direct summand $C$ must also be
injective in $\mathsf{LCA}_{\mathcal{O}}$, \cite[opposite of Cor.
11.6]{MR2606234}. Thus, $C^{\vee}$ is a projective object, but as $C$ is
compact, it is discrete. Thus, $C^{\vee}$ is a projective as an algebraic
$\mathcal{O}$-module. Since $\mathcal{O}$ is a Dedekind domain, the
classification of projective modules yields an isomorphism $C^{\vee}%
\simeq\bigoplus_{J\in\mathcal{I}}J$ for a possibly infinite list $\mathcal{I}$
of ideals (Details: Dedekind domains are hereditary rings, so this claim
follows from \cite[\S 2E, (2.24) Theorem and (2.25)]{MR1653294}). As the
topology on $C^{\vee}$ is discrete, this tautologically also holds in
$\mathsf{LCA}_{\mathcal{O}}$. Thus, by dualizing, $C\simeq\prod_{J\in
\mathcal{I}}\mathbb{T}_{J}$.\newline(2$\Rightarrow$3) The underlying additive
group is $\mathbb{R}^{n}\oplus\mathbb{T}^{\omega}$ for some $n\geq0$ and
$\omega$ some cardinal. By Moskowitz \cite[Theorem 3.2]{MR0215016} this is an
injective object in $\mathsf{LCA}$. Moreover, each $\mathbb{R}_{\sigma}$ is
also an $F$-module, so they are divisible $\mathcal{O}$-modules, thus
injective as $\mathcal{O}$ is Dedekind. Similarly, each $\mathbb{T}_{J}$ is
divisible (use the surjection from the $F$-modules $\mathbb{R}_{\sigma}$ in
Equation \ref{lMinkSeqForIdeals}), and arbitrary products of injective objects
are injective, \cite[opposite of Cor. 11.7]{MR2606234}.\newline(3$\Rightarrow
$2) The underlying additive group is injective in $\mathsf{LCA}$ and thus
using \cite[Theorem 3.2]{MR0215016} isomorphic to $\mathbb{R}^{n}%
\oplus\mathbb{T}^{\omega}$ for some $n\geq0$ and $\omega$ some cardinal in the
category $\mathsf{LCA}$. Dualize to get $G^{\vee}\simeq\mathbb{R}^{n}%
\oplus\bigoplus_{\omega}\mathbb{Z}$. By Lemma
\ref{lemma_VectorGroupDiscreteInLCALiftsToLCAO} this implies that $G^{\vee
}\simeq\bigoplus_{\sigma\in I}\mathbb{R}_{\sigma}\oplus D$ in $\mathsf{LCA}%
_{\mathcal{O}}$ with $D$ discrete. By assumption the underlying algebraic
$\mathcal{O}$-module of $G$ is injective, thus the one of $G^{\vee}$ must be
projective. Thus, so must be its direct summands and hence $D$ is a projective
$\mathcal{O}$-module. Hence, it is a possible infinite direct sum of ideals of
$\mathcal{O}$. Dualizing again, we obtain%
\begin{equation}
G\simeq\bigoplus_{\sigma\in I}\mathbb{R}_{\sigma}\oplus\prod_{J\in\mathcal{I}%
}\mathbb{T}_{J}\text{.} \label{lcis5}%
\end{equation}
(2$\Rightarrow$1) We need to show that $G$ as in Equation \ref{lcis5} is an
injective object in $\mathsf{LCA}_{\mathcal{O}}$. However, a product of
injective objects is always again injective, so it suffices to show that the
factors $\mathbb{R}_{\sigma}$ (for any real or complex place $\sigma$) resp.
$\mathbb{T}_{J}$ for any ideal $J$ are injective in $\mathsf{LCA}%
_{\mathcal{O}}$. Secondly, by Lemma
\ref{lemma_InjectiveObjectsAreDetectedOnCG} it suffices to check this for
sequences $G_{1}\hookrightarrow G_{2}\twoheadrightarrow G_{2}/G_{1}$, where
$G_{2}$ is in $\mathsf{LCA}_{\mathcal{O},cg}$. However, these two cases are
handled by Lemma \ref{lemma_RSigmaAreInjectives} and Lemma
\ref{lemma_TJAreInjectives} respectively. Thus, $G$ is indeed an injective object.
\end{proof}

\begin{acknowledgement}
I heartily thank Dustin Clausen for several discussions; his work started the
entire project and was a great source of inspiration. Moreover, I thank Dmitri
Shakhmatov, Markus Spitzweck, Nikolay Kryuchkov, Martin Levin, and Matthias
Wendt for very helpful correspondence and encouragement. Moreover, we thank
the FRIAS for providing excellent working conditions.
\end{acknowledgement}

\bibliographystyle{amsalpha}
\bibliography{acompat,ollinewbib}

\end{document}